\documentclass[11pt]{amsart}
\usepackage{amsfonts}
\usepackage{amsmath}
\usepackage{amsthm}
\usepackage{amssymb}
\usepackage{mathrsfs}
\usepackage[numbers]{natbib}
\usepackage[fit]{truncate}
\usepackage{hyperref}
\usepackage{theoremref}

\usepackage{xcolor}

\usepackage{tikz-cd}

\usepackage{thmtools}
\usepackage{thm-restate}

\usepackage{fmtcount}

\usepackage{enumitem}
	
\usepackage[UKenglish]{babel}
\usepackage[UKenglish]{isodate}

\usepackage{verbatim}
\usepackage[utf8]{inputenc}

\usepackage{dsfont}

\usepackage{etoolbox}
\makeatletter
\patchcmd{\@maketitle}
{\ifx\@empty\@dedicatory}
{\ifx\@empty\@date \else {\vskip3ex \centering\footnotesize\@date\par\vskip1ex}\fi
	\ifx\@empty\@dedicatory}
{}{}
\patchcmd{\@adminfootnotes}
{\ifx\@empty\@date\else \@footnotetext{\@setdate}\fi}
{}{}{}
\makeatother

\usepackage{catoptions}
\makeatletter

\def\Autoref#1{%
	\begingroup
	\edef\reserved@a{\cpttrimspaces{#1}}%
	\ifcsndefTF{r@#1}{%
		\xaftercsname{\expandafter\testreftype\@fourthoffive}
		{r@\reserved@a}.\\{#1}%
	}{%
		\ref{#1}%
	}%
	\endgroup
}
\def\testreftype#1.#2\\#3{%
	\ifcsndefTF{#1autorefname}{%
		\def\reserved@a##1##2\@nil{%
			\uppercase{\def\ref@name{##1}}%
			\csn@edef{#1autorefname}{\ref@name##2}%
			\autoref{#3}%
		}%
		\reserved@a#1\@nil
	}{%
		\autoref{#3}%
	}%
}
\makeatother

\numberwithin{equation}{section}

\theoremstyle{plain}
\newtheorem{theorem}{Theorem}[section]
\newtheorem{corollary}[theorem]{Corollary}
\newtheorem{lemma}[theorem]{Lem\-ma}
\newtheorem{proposition}[theorem]{Prop\-o\-si\-tion}
\newtheorem{question}[theorem]{Question}
\newtheorem{conjecture}[theorem]{Conjecture}

\newtheorem{fact}[theorem]{Fact}
\newtheorem{observation}[theorem]{Observation}
\theoremstyle{definition}
\newtheorem{definition}[theorem]{Definition}
\newtheorem{construction}[theorem]{Construction}
\newtheorem{remark}[theorem]{Remark}
\newtheorem{example}[theorem]{Example}

 \newtheoremstyle{dotless}{}{}{\itshape}{}{\bfseries}{}{ }{}

\theoremstyle{dotless}
\newtheorem*{query}{\ \ }

\MakeRobust{\ref}

\makeatletter
\newcommand{\labeltext}[2]{%
	\@bsphack
	\csname phantomsection\endcsname 
	\def\@currentlabel{#1}{\label{#2}}%
	\@esphack
}
\makeatother

\newenvironment{acknowledgements}{\ 
	
	{\textsl{Acknowledgements.}}}{}

\newcommand{\N}[0]{\mathbb{N}}
\newcommand{\Z}[0]{\mathbb{Z}}
\newcommand{\Q}[0]{\mathbb{Q}}
\newcommand{\R}[0]{\mathbb{R}}

\renewcommand{\H}[0]{\mathbb{H}}
\newcommand{\OO}[0]{\mathcal{O}}
\newcommand{\MM}[0]{\mathcal{M}}
\newcommand{\NN}[0]{\mathcal{N}}
\newcommand{\supp}[0]{\mathrm{supp}}

\newcommand{\brackets}[1]{\left( #1 \right)}
\newcommand{\setbr}[1]{\left\{ #1 \right\}}
\newcommand{\pow}[1]{\!\left(\!\left( #1 \right)\!\right)}

\renewcommand{\L}[0]{\mathcal{L}}
\newcommand{\Lor}[0]{\mathcal{L}_{\mathrm{or}}}
\newcommand{\Log}[0]{\mathcal{L}_{\mathrm{og}}}
\newcommand{\Lr}[0]{\mathcal{L}_{\mathrm{r}}}
\newcommand{\Lvf}[0]{\mathcal{L}_{\mathrm{vf}}}

\newcommand{\Trcf}[0]{T_{\mathrm{rcf}}}

\newcommand{\rc}[1]{{#1}^{\mathrm{rc}}}
\renewcommand{\div}[1]{{#1}^{\mathrm{div}}}

\newcommand{\ol}[1]{\overline{#1}}
\newcommand{\ul}[1]{\underline{#1}}
\newcommand\restr[2]{{
		\left.\kern-\nulldelimiterspace 
		#1
		\vphantom{\big|} 
		\right|_{#2}
}}
\newcommand{\vmin}[0]{v_{\min}}
\newcommand{\vnat}[0]{v_{\mathrm{nat}}}

\DeclareMathOperator{\Char}{char}
\DeclareMathOperator{\ac}{ac}
\DeclareMathOperator{\cl}{cl}
\DeclareMathOperator{\dcl}{dcl}
\DeclareMathOperator{\Th}{Th}
\DeclareMathOperator{\dprk}{dp-rk}

\DeclareMathOperator{\td}{td}

\begin{document}

	\begin{abstract}
		We investigate what henselian valuations on ordered fields are definable in the language of ordered rings. This leads towards a systematic study of the class of ordered fields which are dense in their real closure. Some results have connections to recent conjectures on definability of henselian valuations in strongly NIP fields. Moreover, we obtain a complete characterisation of strongly NIP almost real closed fields.
	\end{abstract}

	\title[On Strongly NIP Ordered Fields]{On Strongly NIP Ordered Fields and Definable Convex Valuations}
	
	\author[L.~S.~Krapp]{Lothar Sebastian Krapp}
	\author[S.~Kuhlmann]{Salma Kuhlmann}
	\author[G.~Lehéricy]{Gabriel Lehéricy}
	
	\address{Fachbereich Mathematik und Statistik, Universität Konstanz, 78457 Konstanz, Germany}
	\email{sebastian.krapp@uni-konstanz.de}
	
	\address{Fachbereich Mathematik und Statistik, Universität Konstanz, 78457 Konstanz, Germany}
	\email{salma.kuhlmann@uni-konstanz.de}
	
	\address{École supérieure d'ingénieurs Léonard-de-Vinci, Pôle Universitaire Lé\-o\-nard de Vinci, 92 916 Paris La Défense Cedex, France}
	\email{gabriel.lehericy@imj-prg.fr}
	
	\date{5 February 2019}
	
	\cleanlookdateon
	
	\maketitle
	
	\tableofcontents

	\section{Introduction}
	
	Let $\L_{\mathrm{r}} = \{+,-,\cdot,0,1\}$ be the language of rings, $\L_{\mathrm{or}} = \L_{\mathrm{r}} \cup \{<\}$  the language of ordered rings and $\L_{\mathrm{og}} = \{+,0,<\}$ the language of ordered groups. Throughout this work, we will abbreviate the $\L_{\mathrm{r}}$-structure of a field $(K,+,-,\cdot,0,1)$ simply by $K$, the $\L_{\mathrm{or}}$-structure of an ordered field $(K,+,-,\cdot,0,1,<)$ by $(K,<)$ and the $\L_{\mathrm{og}}$-structure of an ordered group $(G,+,0,<)$ by $G$.
	
	The following conjecture is due to Shelah--Hasson (see \cite{shelah,dupont,halevi}):

	\begin{conjecture}\thlabel{conj:shelah-hasson}
		Let $K$ be an infinite strongly NIP field. Then $K$ is either real closed, or algebraically closed, or admits a non-trivial $\L_{\mathrm{r}}$-definable\footnote{Throughout this work \emph{definable} always means \emph{definable with parameters}. } henselian valuation.
	\end{conjecture}
	
	By adapting the characterisation of dp-minimal fields in \cite{johnson}, Halevi, Hasson and Jahnke \cite{halevi} obtain a conjectural classification of strongly NIP fields in the language $\Lr$ which is equivalent to \thref{conj:shelah-hasson} (cf. \cite[Conjecture~1.3]{halevi}).

	Note that in ordered fields, henselian valuations are always convex by the following fact.
	
	\begin{fact}\emph{(See \cite[Lemma~2.1]{knebusch}.)}\thlabel{fact:hensconv}
		Let $(K,<)$ be an ordered field and $v$ a henselian valuation on $K$. Then $v$ is convex on $(K,<)$.
	\end{fact}

	We specialise \thref{conj:shelah-hasson} to ordered fields and enhance it as follows.

	\begin{restatable}{conjecture}{mainconjecture}
	\thlabel{conj:main}
		Let $(K,<)$ be a strongly NIP ordered field. Then $K$ is either real closed or admits a non-trivial $\Lor$-definable henselian valuation.
	\end{restatable}
	
	
	Note that there are ordered fields which admit a non-trivial $\Lor$-definable henselian valuation but are not NIP (see \thref{rmk:mainrmk}~(\ref{rmk:mainrmk:2})). 
	\thref{conj:main} motivates the study of non-trivial $\Lor$-definable henselian valuations on a given ordered field. Moreover, we can reformulate \thref{conj:main} in terms of the model theoretically well-studied class of almost real closed fields (cf. \cite{delon}, see also \thref{def:arc}). 
		
	\begin{restatable}{conjecture}{classificationconjecture}
		\thlabel{conj:classification}
		Any strongly NIP ordered field is almost real closed.
	\end{restatable}

	In \Autoref{sec:prelim} we gather some general basic preliminaries. In \Autoref{sec:densdivhull} we start with some preliminaries on ordered abelian groups with particular focus on algebraic and valuation theoretic properties of ordered abelian groups which are dense in their divisible hull. Results of this section are mainly used in \Autoref{sec:defconvval} and \Autoref{sec:density}. In \Autoref{sec:defconvval} we study $\Lor$-definable henselian valuations in ordered fields. The main result of \Autoref{subsec:def} gives sufficient conditions on the residue field and the value group of a henselian valuation $v$ in order that $v$ is $\Lor$-definable (see \thref{thm:defval}). We then proceed by comparing \thref{thm:defval} to known $\Lr$-definability results of henselian valuations.  Special emphasis is put on definable valuations in almost real closed fields in \Autoref{sec:arc}. In this subsection, we firstly prove analogues in the language of ordered rings to known model theoretic results about almost real closed fields in the language of rings, and secondly we compare $\Lr$- and $\Lor$-definability of henselian valuations in almost real closed fields. The results of \Autoref{sec:densdivhull} on ordered abelian groups which are dense in their divisible hull and the conditions of \thref{thm:defval} motivate the study of ordered fields which are dense in their real closure. In \Autoref{sec:density} we will show that the property of an ordered field to be dense in its real closure is preserved under elementary equivalence (see \thref{thm:densitytransfers}). Our investigation of strongly NIP ordered fields starts in \Autoref{sec:conjecturalclassification}. We begin by focussing on dp-minimal ordered fields (which are, in particular, strongly NIP). By careful analysis of the results of \cite{jahnke}, we deduce in \thref{prop:dpminarc} that an ordered field is dp-minimal if and only if it is almost real closed with respect to some dp-minimal ordered abelian group $G$. Thereafter, we address the following query (\textbf{classification of strongly NIP ordered fields}): 
	
	\begin{query}
		An ordered field is strongly NIP if and only if it is almost real closed with respect to some strongly NIP ordered abelian group $G$. \hfill\stepcounter{equation}\emph{(\decimal{section}.\decimal{equation})}\labeltext{(\decimal{section}.\decimal{equation})}{query}
	\end{query}
	In \thref{thm:arcf} we show that an almost real closed field with respect to some ordered abelian group $G$ is strongly NIP if and only if $G$ is strongly NIP. This settles the backward direction of the above query, and reduces its forward direction to \thref{conj:classification}.
	In \Autoref{sec:conclusion}, we show that \thref{conj:main} and \thref{conj:classification} are equivalent (see \thref{thm:main}).
	In \Autoref{sec:questions}, we conclude by stating some open questions motivated by this work.
	
	Except for minor cross-references, \Autoref{sec:conjecturalclassification} and \Autoref{sec:conclusion} do not rely on the other sections and can be read independently.
	
	\section{General Preliminaries}\label{sec:prelim}

	All notions on valued fields and groups can be found in \cite{kuhlmann,engler} and all notions on strongly NIP theories in \cite{simon}. The set of natural numbers with $0$ will be denoted by $\N_0$, the set of natural numbers without $0$ by $\N$.
	
	Let $K$ be a field and $v$ a valuation on $K$. We denote the \textbf{valuation ring} of $v$ in $K$ by $\OO_v$, the \textbf{valuation ideal}, i.e.  the maximal ideal of $\OO_v$, by $\MM_v$, the \textbf{ordered value group} by $vK$ and the \textbf{residue field} $\OO_v/\MM_v$ by $Kv$. For $a \in \OO_v$ we also denote $a + \MM_v$ by $\ol{a}$. For an ordered field $(K,<)$ a valuation is called \textbf{convex} (in $(K,<)$) if the valuation ring $\OO_v$ is a convex subset of $K$. In this case, the relation $\ol{a} < \ol{b} : \Leftrightarrow \ol{a} \neq \ol{b} \wedge a < b$ defines an order relation on $Kv$ making it an ordered field. 
	
	Let $\Lvf = \Lr\cup \{\OO_v\}$ be the \textbf{language of valued fields}, where $\OO_v$ stands for a unary predicate. Let $(K,\OO_v)$ be a valued field. An atomic formula of the form $v(t_1) \geq v(t_2)$, where $t_1$ and $t_2$ are $\Lr$-terms, stands for the $\Lvf$-formula $t_1=t_2=0 \vee (t_2\neq 0 \wedge \OO_v(t_1/t_2))$. Thus, by abuse of notation, we also denote the $\Lvf$-structure $(K,\OO_v)$ by $(K,v)$. Similarly, we also call $(K,<,v)$ an ordered valued field. 
	We say that a valuation $v$ is $\L$-definable for some language $\L \in \{\Lr,\Lor\}$ if its valuation ring is an $\L$-definable subset of $K$.
	
	Let $K$ be a field and $v,w$ be valuations on $K$. We write $v\leq w$ if and only if $\OO_v \supseteq \OO_w$. In this case we say that $w$ is \textbf{finer than $v$} and $v$ is \textbf{coarser than $v$}. If $\OO_v \supsetneq \OO_w$, we write $v< w$ and say that $w$ is \textbf{strictly finer} than $v$ and that $v$ is \textbf{strictly coarser} than $w$. Note that $\leq$ defines an order relation on the set of convex valuations of an ordered field. 
	We call two elements $a,b \in K$ \textbf{archimedean equivalent} (in symbols $a \sim b$) if there is some $n\in \N$ such that $|a|<n|b|$ and $|b|<n|a|$. Let $G = \{[a] \mid a \in K^\times\}$ the set of archimedean equivalence classes of $K^\times$. Equipped with addition $[a]+[b] = [ab]$ and the ordering $[a] < [b] : \Leftrightarrow a \not\sim b \wedge |b| < |a|$, the set $G$ becomes an ordered abelian group. Then $v: K^\times \to G$ defines a convex valuation on $K$. This is called the \textbf{natural valuation} on $K$, denoted by $\vnat$. We say that an extension of ordered fields $(K,<) \subseteq (L,<)$ is \textbf{immediate} if it is immediate\footnote{For the definition of an immediate extension of valued fields, see \cite[p.~27]{kuhlmann}} with respect to the natural valuation. The extension is \textbf{dense} if $K$ is dense in $L$.
	
	Let $(k,<)$ be an ordered field and $G$ an ordered abelian group. We denote the \textbf{ordered Hahn field} with coefficients in $k$ and exponents in $G$ by $k\pow{G}$. We denote an element $s \in k\pow{G}$ by $s = \sum_{g \in G} s_gt^g$, where $s_g = s(g)$ and $t^g$ is the characteristic function on $G$ mapping $g$ to $1$ and everything else to $0$. The ordering on $k\pow{G}$ is given by $s > 0 : \Leftrightarrow s(\min \supp s) > 0$, where $\supp s = \{g \in G \mid s(g) \neq 0\}$ is the \textbf{support of $s$}. Let $\vmin$ be the valuation on $k\pow{G}$ given by $\vmin(s) = \min \supp s$ for $s \neq 0$. Note that $\vmin$ is convex and henselian. Note further that if $k$ is archimedean, then $\vmin$ coincides with $\vnat$.
	
	We will repeatedly use the Ax--Kochen--Ershov principle for ordered fields (cf. \cite[Corollary~4.2(iii)]{farre}).
	
	\begin{fact}[Ax--Kochen--Ershov principle]
		Let $(K, <,v)$ and $(L,<,w)$ be two henselian ordered valued fields. Then $(Kv,<) \equiv (Lw,<)$ and $vK \equiv wL$ if and only if $(K,<, v) \equiv (L,<, w)$.
	\end{fact} 
	
	Let $\L$ be a language and $T$ an $\L$-theory. We fix a monster model $\MM$ of $T$. Let $\varphi(\ul{x};\ul{y})$ be an $\L$-formula. We say that \textbf{$\varphi$ has the independence property (IP)} if there are $(\ul{a}_i)_{i \in \omega}$ and $(\ul{b}_J)_{J \subseteq \omega}$ in $M$ such that $\MM \models \varphi(\ul{a}_i;\ul{b}_J)$ if and only if $i \in J$. We say that the theory \textbf{$T$ has IP} if there is some formula $\varphi$ which has IP. If $T$ does not have IP, it is called \textbf{NIP} (\emph{not the independence property}). For an $\L$-structure $\NN$, we also say that \textbf{$\NN$ is NIP} if its complete theory $\Th(\NN)$ is NIP. A well-known example of an IP theory is the complete theory of the $\Lr$-structure $(\Z,+,-,\cdot,0,1)$ (cf. \cite[Example~2.4]{simon}). Since $\Z$ is parameter-free definable in the $\Lr$-structure $\Q$ (cf. \cite[Theorem~3.1]{robinson2}), also the complete $\Lr$-theory of $\Q$ has IP.
	
	Let $A \subseteq M$ be a set of parameters, $\Delta$ a set of $\L$-formulas and $(J,<)$ a linearly ordered set. A sequence $S = (a_j \mid j\in J)$ in $M$ is \textbf{$\Delta$-indiscernible over $A$} if for every $k \in \N$, any increasing tuples $i_1<\ldots<i_k$ and $j_1<\ldots<j_k$ in $J$, any formula $\varphi({x_1},\ldots,{x_k};\ul{y})\in \Delta$ and any tuple $\ul{b} \in A$, we have $\MM\models \varphi(a_{i_1},\ldots,a_{i_k};\ul{b}) \leftrightarrow \varphi(a_{j_1},\ldots,a_{j_k};\ul{b})$. The sequence $S$ is called \textbf{indiscernible over $A$} if it is $\Delta$-indiscernible over $A$ for any set of $\L$-formulas $\Delta$. A family of sequences $(S_t \mid t \in X)$ is called \textbf{mutually indiscernible over $A$} if for each $u \in X$, the sequence $S_u$ is indiscernible over $A \cup \bigcup_{t \in X \setminus\{u\}}S_t$.
	
	Let $p$ be a partial $n$-type over a set $A \subseteq M$. We define the \textbf{dp-rank of $p$ over $A$} as follows: Let $\kappa$ be a cardinal. The dp-rank of $p$ over $A$ is less than $\kappa$ (in symbols, $\dprk(p,A) < \kappa$) if for every family $(S_t \mid t< \kappa)$ of mutually indiscernible sequences over $A$ and any $\ul{b} \in M^n$ realising $p$ in $\MM$, there is some $t < \kappa$ such that $S_t$ is indiscernible over $A \cup \{b_1,\ldots,b_n\}$. The theory $T$ is called \textbf{strongly NIP} if it is NIP and $\dprk(\{x=x\},\emptyset) < \aleph_0$, where $\{x=x\}$ is the partial type over $\emptyset$ only consisting of the formula $x = x$. The theory $T$ is called \textbf{dp-minimal} if it is NIP and $\dprk(\{x=x\},\emptyset) = 1$.
	Again, we call an $\L$-structure $\NN$ strongly NIP (respectively dp-minimal) if $\Th(\NN)$ is strongly NIP (respectively dp-minimal). 
	
	Any reduct of a strongly NIP structure is strongly NIP (cf. \cite[Claim~3.14, 3)]{shelah2}) and any reduct of a dp-minimal structure is dp-minimal (cf. \cite[Observation~3.7]{onshuus}).
	Since any weakly o-minimal theory is dp-minimal (cf. \cite[Corollary~4.3]{dolich}), we obtain the following hierarchy:
	
	\textbf{o-minimal $\to$ weakly o-minimal $\to$ dp-minimal $\to$ strongly NIP $\to$ NIP}
	
	In particular, any divisible ordered abelian group and any real closed field are strongly NIP.

	\section{Preliminaries on Ordered Abelian Groups}\label{sec:densdivhull}
	
	In this section, we will present some preliminary results on ordered abelian groups which will be used throughout the other sections.
	
	Let $G$ and $H$ be ordered abelian groups with $G \subseteq H$. We say that \textbf{$G$ is dense in $H$} or that the extension $G \subseteq H$ is dense if for any $a,b \in H$ there is $c \in G$ such that $a<c<b$.
	If $\Z$ is a convex subgroup of $G$, then we say that $G$ is \textbf{discretely ordered}. Throughout this section, we denote the \textbf{natural valuation} on a given ordered abelian group by $v$.\footnote{See \cite[p.~9]{kuhlmann} for a definition of the natural valuation, noting that an ordered abelian group is an ordered $\Z$-module.} The value set $vG$ is ordered by $v(g_1) < v(g_2)$ if $v(g_1) \neq v(g_2)$ and $|g_1| > |g_2|$. We say that an extension of ordered abelian groups $G \subseteq H$ is \textbf{immediate} if it is immediate with respect to the natural valuation.\footnote{See \cite[p.~3]{kuhlmann} for a definition of an immediate extension.}
	Let $\gamma \in vG$, and let $G^\gamma$ and $G_\gamma$ be the following convex subgroups of $G$: $$G^\gamma = \setbr{g \in G \mid v(g) \geq \gamma} \text{ and } G_\gamma = \setbr{g \in G \mid v(g) > \gamma}.$$
	The \textbf{archimedean component} $B(G,\gamma)$ of $G$ corresponding to $\gamma$ is given by $B(G,\gamma) = G^\gamma/G_\gamma$. If no confusion arises, we will only write $B_\gamma$ instead of $B(G,\gamma)$. Note that $B_\gamma$ is an ordered group with the order induced by $G$.
	A valuation $w$ on $G$ is \textbf{convex} if for any $g_1,g_2 \in G$ with $0<g_1\leq g_2$, we have $w(g_1) \geq  w(g_2)$. Note that there is a one-to-one correspondence between non-trivial convex subgroups of $G$ and final segments of $vG$ (see \cite[p.~50~f.]{kuhlmann}).
	
	We say that $G$ has a \textbf{left-sided limit point} $g_0$ in $H$ if for any  $g_1 \in H$ with $g_1>0$ the intersection of $(g_0-g_1,g_0)$ with $G$ is non-empty. Similarly, $g_0$ is a \textbf{right-sided limit point} if for any $g_1 \in H$ with $g_1>0$ the intersection of $(g_0,g_0+g_1)$ with $G$ is non-empty. A \textbf{limit point} is a point which is a left-sided or a right-sided limit point.
	The \textbf{divisible hull} of $G$ is denoted by $\div{G}$. Note that $G$ and $\div{G}$ have the same rank, i.e. $vG = v\div{G}$. The \textbf{closure} of $G$ in $\div{G}$ with respect to the order topology is denoted by $\cl(G)$. Note that $G$ is dense in $\div{G}$ if and only if $\cl(G) = \div{G}$. Note further that $G$ has a limit point in $\div{G} \setminus G$ (i.e. there is some $a \in \div{G} \setminus G$ such that $a$ is a limit point of $G$ in $\div{G}$) if and only if $G$ is not closed in $\div{G}$.
	
	For any ordered abelian groups $G_1$ and $G_2$, we denote the \textbf{lexicographic sum} of $G_1$ and $G_2$ by $G_1 \oplus G_2$. This is the abelian group $G_1 \times G_2$ with the lexicographic ordering $(a,b) < (c,d)$ if $a<c$, or $a=c$ and $b<d$. Let $(\Gamma,\leq)$ be an ordered set and for each $\gamma \in \Gamma$, let $A_\gamma$ be an archimedean ordered abelian group. For any element $s$ in the product group $\prod_{\gamma \in \Gamma} A_\gamma$, define the \textbf{support of $s$} by $\supp(s) = \{\gamma \in \Gamma \mid s(\gamma) \neq 0\}$. The \textbf{Hahn product} $\H_{\gamma \in \Gamma}A_\gamma$ is the subgroup of $\prod_{\gamma \in \Gamma} A_\gamma$ consisting of all elements with well-ordered support. Moreover, $\H_{\gamma \in \Gamma}A_\gamma$ becomes an ordered group under the order relation $s > 0 :\Leftrightarrow s(\min\supp (s)) > 0$. We express elements $s$ of $\H_{\gamma \in \Gamma}A_\gamma$ by $s = \sum_{\gamma \in \Gamma} s_\gamma \mathds{1}_\gamma$, where $s_\gamma = s(\gamma)$ and $\mathds{1}_\gamma$ is the characteristic function of $\gamma$ mapping $\gamma$ to $1$ and everything else to $0$. The \textbf{Hahn sum} $\coprod_{\gamma \in \Gamma} A_\gamma$ is the ordered subgroup of $\H_{\gamma \in \Gamma}A_\gamma$ consisting of all elements with finite support.
	\\

	In \Autoref{sec:defconvval}, we address the question what convex valuations are $\Lor$-definable in ordered fields. In analogy to \thref{constr:val}, we will show in \thref{prop:subgroupdef} that for any densely ordered abelian group which is not dense in its divisible hull, there exists a proper non-trivial convex $\Log$-definable subgroup. We will repeatedly use the following characterisation of densely ordered abelian groups.
	
	\begin{proposition}\thlabel{prop:denseeq}
		Let $G$ be an ordered abelian group. Then the following are equivalent:
		
		\begin{enumerate}[wide, labelwidth=!, labelindent=6pt]
			\item \label{prop:denseeq:1} $G$ is densely ordered.
			
			\item \label{prop:denseeq:2} $0$ is a limit point of $G$ in $G$.
			
			\item \label{prop:denseeq:3} $0$ is a limit point of $G$ in $\div{G}$.
			
			\item  \label{prop:denseeq:4} Either $vG$ has no last element, or $vG$ has a last element $\gamma$ and $B_\gamma$ is densely ordered.
		\end{enumerate} 
	\end{proposition}
	
	\begin{proof}
		We may assume that $G \neq \{0\}$, as in the case $G = \{0\}$ the e\-quiv\-a\-lences are clear.
		The equivalence of (\ref{prop:denseeq:1}) and (\ref{prop:denseeq:2}) is an easy consequence of the definition of a dense order on an abelian group. 
		
		Obviously, (\ref{prop:denseeq:3}) implies (\ref{prop:denseeq:2}). For the converse, suppose that (\ref{prop:denseeq:1}) holds. Let $g \in G\setminus\setbr{0}$ and $N \in \N$. We need to find $c \in G$ with $0 < c < \frac {|g|} N$. Since $G$ is densely ordered, there are $c_1,\ldots,c_N \in G$ such that $0<c_1<\ldots<c_N <|g|$. Let $c = \min\setbr{c_{i+1}-c_i \mid i \in \setbr{0,\ldots,N}}$, where we set $c_0 = 0$ and $c_N = |g|$. We obtain $0<c<2c<\ldots<Nc<|g|$ and thus $0<c<\frac{|g|}N$, as required.
		
		It remains to show that (\ref{prop:denseeq:1}) and (\ref{prop:denseeq:4}) are equivalent. Suppose that $vG$ has a last element $\gamma$ and $B_\gamma$ is not densely ordered. By (\ref{prop:denseeq:2}) applied to $B_\gamma$, there is some $g \in G$ with $g>0$ and $v(g) = \gamma$ such that there is no element in $B_\gamma$ strictly between $0 + G_\gamma$ and $g + G_\gamma$. Let $h\in G$ such that $0 < h \leq g$. Then $v(h) \geq v(g) = \gamma$. By maximality of $\gamma$, we have $v(h) = \gamma$. Thus, $0 + G_\gamma < h + G_\gamma \leq g + G_\gamma$. This implies $h + G_\gamma = g + G_\gamma$, i.e. $v(g-h) > \gamma$. Again, by maximality of $\gamma$, we obtain $g = h$. Hence, there is no element in $G$ stricly between $0$ and $g$, showing that $G$ is not densely ordered. This shows that (\ref{prop:denseeq:1}) implies (\ref{prop:denseeq:4}). For the converse, suppose that (\ref{prop:denseeq:4}) holds. We will show that $0$ is a limit point of $G$ in $G$. Let $g \in G$ with $g > 0$. If $v(g)$ is not maximal, then let $h \in G$ with $h > 0$ and $v(h) > v(g)$. Then $0<h<g$, as required. Otherwise, $\gamma = v(g)$ is the maximum of $vG$, and by assumption, $B_\gamma$ is densely ordered. Thus, there is some $h \in G$ with $v(h) = \gamma$ such that $0 + G_\gamma < h + G_\gamma < g + G_\gamma$, whence $0 < h < g$.
	\end{proof}
	
	\begin{proposition}\thlabel{prop:subgroupdef}
		Let $G$ be a densely ordered abelian group which is not dense in $\div{G}$.
		Then $G$ has a proper non-trivial convex subgroup which is $\Log$-definable with one parameter.
	\end{proposition}
	
	\begin{proof}
		Let $g_0 \in \div{G} \setminus \cl(G)$ with $g_0>0$. If $\gamma = v(g_0)$ is the maximum of $vG$, then let $h \in G$ with $v(h) < \gamma$. Note that $v(g_0 + h) = v(h) < \gamma$ and that $g_0 + h$ is not a limit point of $G$ in $\div{G}$, as $g_0$ is not a limit point of $G$ in $\div{G}$. Hence, $g_0 + h \in \div{G} \setminus \cl(G)$. Hence, by replacing $g_0$ by $g_0 + h$ if necessary, we may assume that $v(g_0)$ is not the maximum of $vG$.
		
		Let $g_1 \in G$ with $g_1>0$ and $N\in \N$ such that $g_0 = \frac{g_1}N$. Consider the set 
		$$D = \setbr{g \in G^{\geq 0} \mid g < g_0} = \setbr{g \in G^{\geq 0} \mid Ng < g_1}.$$ This set is $\Log$-definable with the parameter $g_1$. Let $$A = \setbr{g \in G^{\geq 0} \mid g + D \subseteq D}.$$
		Again, $A$ is $\Log$-definable with the parameter $g_1$. Note that $A$ is convex. Obviously, $A \neq G^{\geq0}$, as for any $g \in D$ and $h \in G^{\geq0}$ with $h>g_0-g$ we have $g + h > g_0$ and thus $h \notin A$. Assume that $A = 0$. Then for any $g \in G^{> 0}$, let $g' \in G^{\geq 0}$ with $g' < g_0$ and $g + g' > g_0$ and set $c_g = g + g'$. Then $g_0 < c_g < g_0 + g$.
		Since by \thref{prop:denseeq}~(\ref{prop:denseeq:3}), $0$ is a limit point of $G$ in $\div{G}$, for any $h \in \div{G}$ with $h > 0$, there exists some $g \in G$ such that  $0<g<h$. Thus, $g_0< c_g < g_0 + g < g_0 + h$. This shows that $g_0$ is a limit point of $G$ in $\div{G}$, contradicting the choice of $g_0 \notin \cl(G)$. Hence, $ A \neq 0$.
		
		Now for any $a,b \in A$ with $0<a<b$, we have $(a + b) + D \subseteq a + D \subseteq D$, whence $a+b \in A$. Moreover, $0<b-a<b$ and $-b<a-b<0$. Thus, by convexity, $b-a \in A$ and $a-b \in -A$. Similarly, for any $a,b \in -A$ with $a<b<0$, we have $a\pm b \in -A$ and $b-a \in A$. This shows that $H = -A \cup A$ is closed under addition and thus a $\Log$-definable convex subgroup of $G$. Since $0 \neq A \neq G^{\geq 0}$, we also have that $H$ is a proper non-trivial subgroup of $G$.
	\end{proof}
	
	We now consider regular ordered abelian groups. This well-known class of ordered abelian groups has been studied model theoretically in \cite{robinson} and algebraically in \cite{zakon}. 	An ordered abelian group is \textbf{regular} if it satisfies one of the equivalent conditions in \thref{fact:defregular} (cf. \cite[p.~14]{hong}, \cite[p.~137]{fehm} and \cite[p.~1148]{delon}).
	
	\begin{fact}\thlabel{fact:defregular}
		Let $G$ be an ordered abelian group. Then the following are equivalent:
		\begin{enumerate}[wide, labelwidth=!, labelindent=6pt]
			\item \label{fact:defregular1} For any prime $p \in \N$ and for any infinite convex subset $A \subseteq G$, there is a $p$-divisible element in $A$.
			
			\item \label{fact:defregular2} For any $n\in \N$ and any $a,b \in G$, if there are $g_1,\ldots,g_n \in G$ with $a \leq g_1 < \ldots < g_n \leq b$, then there  is some $c \in G$ with $a \leq nc \leq b$.
			
			\item For any non-trivial convex subgroup $H \subseteq G$, the quotient group $G/H$ is divisible.
		\end{enumerate}
	\end{fact}
	
	The following fact is due to \cite{robinson} and \cite{zakon}.
	
	\begin{fact}\thlabel{fact:regrobinson}
		Let $G$ be an ordered abelian group. Then the following hold:
		
		\begin{enumerate}[wide, labelwidth=!, labelindent=6pt]
			\item \label{fact:regrobinson:1} $G$ is regular if and only if it has an archimedean model, i.e. there is some archimedean ordered abelian group $H$ such that $H \equiv G$.
			
			\item $G$ is regular and discretely ordered if and only if it is a $\Z$-group, i.e. $G \equiv \Z$ as ordered groups.
		\end{enumerate}
	\end{fact}

	We can deduce from \thref{fact:regrobinson}~(\ref{fact:regrobinson:1}) that any Hahn product whose rank has no last element and which is not divisible does not have an archimedean model. A similar result holds for ordered Hahn fields (see \thref{rmk:hahnfieldnoarch}).
	
	\begin{proposition}\thlabel{prop:hahnnoarch}
		Let $(\Gamma,<)$ be a totally ordered set without a last element, and for any $\gamma \in \Gamma$, let $A_\gamma \neq \{0\}$ be an ordered abelian group. Suppose that the Hahn product $G = \H_{\gamma \in \Gamma}  A_\gamma$ is non-divisible. Then $G$ has no archimedean model.
	\end{proposition}

	\begin{proof}
		By  \thref{fact:regrobinson}~(\ref{fact:regrobinson:1}), we need to show that $G$ is not regular. Since $G$ is non-divisible, there is some $\gamma_0 \in \Gamma$ such that $A_{\gamma_0}$ is non-divisible. Then for $H= \H_{\gamma \in \Gamma^{>\gamma_0}} A_\gamma$ we have $G / H \cong \H_{\gamma \in \Gamma^{\leq \gamma_0}} A_\gamma$, which is not divisible. Hence, $G$ is not regular.
	\end{proof}

	In \Autoref{sec:density}, we will study the class of ordered fields which are dense in their real closure. The analogue in the ordered abelian group case is the class of ordered abelian groups which are dense in their divisible hull. The following proposition shows that this is exactly the class of regular densely ordered abelian groups.
	
	\begin{proposition}\thlabel{prop:regimpliesdense}
		Let $G$ be an ordered abelian group. Then $G$ is regular and densely ordered if and only if $G$ is dense in $\div{G}$.
	\end{proposition}
	
	\begin{proof}
		Since for divisible ordered abelian groups the conclusion in trivial, we assume that $G$ is non-divisible. 
		
		Suppose that $G$ is regular and densely ordered. Let $a,b \in G$ and $n \in \N$ such that $\frac an < \frac bn$. Since $G$ is densely ordered, there are $g_1,\ldots,g_n \in G$ with $a \leq g_1 < \ldots < g_n \leq b$. By \thref{fact:defregular}~(\ref{fact:defregular2}), there  is some $c \in G$ with $a \leq nc \leq b$. Hence, $\frac an < c < \frac bn$, as required.
		
		Conversely, suppose that $G$ is dense in $\div{G}$. Then $G$ is densely ordered as $\div{G}$ is densely ordered. Let $p \in \N$ be prime, $A \subseteq G$ an infinite convex subset of $G$ and $a,b \in A$ with $a<b$. Then there is some $g \in G$ with $\frac a p < g < \frac bp$ and thus $a < pg < b$. Hence, $A$ contains a $p$-divisible element, as required by \thref{fact:defregular}~(\ref{fact:defregular1}).
	\end{proof}
	
	\begin{remark}\thlabel{rmk:denseindivregular}
	Note that \thref{fact:regrobinson} and \thref{prop:regimpliesdense} imply that an ordered abelian group is dense in its divisible hull if and only if it has a densely ordered archimedean model. 
	\end{remark}
	
	The following two examples of ordered abelian groups which are not regular will be used in \Autoref{sec:defconvval} for a comparison between $\Lor$- and $\Lr$-definability of a given henselian valuation.
	
	\begin{example}\thlabel{ex:regnew}
		\begin{enumerate}[wide, labelwidth=!, labelindent=6pt]
			\item  \label{ex:regnew:1}
			$\Z \oplus \Z$ is a discretely ordered group which is not regular.
			
			\item  \label{ex:regnew:2} Let	$A = \setbr{\left.\frac{a}{2^n} \ \right| a, n \in \Z}$.
			Note that $A$ is dense in $\div{A} = \Q$. Consider $G = A \oplus A$. Since there is no element in $G$ between $\left(\frac 1 3 , 0\right)$ and $\left(\frac 1 3, 1\right)$, which both lie in $\div{G} = \Q \oplus \Q$, it follows that $G$ is not dense in $\div{G}$ and thus not regular. However, $G$ has limit points in $\div{G}\setminus G$, e.g. $\left(0,\frac 1 3\right)$.
			
		\end{enumerate}
	\end{example} 
	
	It is known that any dense extension of ordered fields is immediate (cf. \cite[p.~29~f.]{kuhlmann}). In the following we will study extensions of ordered abelian groups $G \subseteq H$ with special focus on the case that $H = \div{G}$. Note that for any $\gamma \in vG$, we have $B(\gamma,\div{G}) = \div{B(\gamma,G)}$. Hence, the extension $G \subseteq \div{G}$ is immediate if and only if all archimedean components of $G$ are divisible.
	
	\begin{proposition}\thlabel{prop:densimpimm}
		Let $G\subseteq H$ be a dense extension of ordered abelian groups. Suppose that $vG$ has no last element. Then the extension $G \subseteq H$ is immediate.
	\end{proposition}
	
	\begin{proof}
		In order to show that $G \subseteq H$ is immediate, we need to show that for any $a \in H$ there exists $b \in G$ such that $v(a-b)> v(a)$. Note that the density of $G$ in $H$ implies $vG = vH$. Let $a \in H$ with $a>0$. Since $vG$ has no last element, there is some $c \in G^{>0}$ such that $v(c) > v(a)$. By density of $G$ in $H$, there is some $b \in G$ such that $a-c< b < a+c$. We obtain $v(a-b) \geq  v(c) > v(a)$, as required.
	\end{proof}
	
	\begin{corollary}\thlabel{cor:densimpimm}
		Let $G$ be an ordered abelian group such that $vG$ has no last element and $G$ is dense in $\div{G}$. Then the extension $G \subseteq \div{G}$ is immediate.
	\end{corollary}
	
	\begin{proof}
		Apply \thref{prop:densimpimm} to $H = \div{G}$.
	\end{proof}
	
	\thref{cor:densimpimm} does not hold in general in the case where $vG$ has a last element, as the following example will show.
	
	\begin{example}\thlabel{ex:nonarchdense}
		 Let $A$ be as in \thref{ex:regnew}~(\ref{ex:regnew:2}). For $G = \Q \oplus A$ we have $\div{G} = \Q \oplus \Q$. Moreover, $G$ is dense in $\div{G}$, as $A$ is dense in $\Q$. However, the extension is not immediate, as the archimedean components do not coincide.
	\end{example}

	We will now consider the converse direction, that is, under what conditions an immediate extension of ordered abelian groups is dense.
	
	\begin{lemma}\thlabel{lem:denseimm}
		Let $G,A_1,A_2,\ldots$ be ordered abelian groups such that $G\subseteq H = \H_{n \in \omega} A_n$ is immediate. Let $a \in \H_{n \in \omega} A_n$. Then there exists a sequence $(d_n)_{n \in \omega}$ in $G$ such that for any $k \in \omega$ and any $i \leq k$ we have $(d_0+\ldots+d_k)(i) = a(i)$.
	\end{lemma}
	
	\begin{proof}
		Since $G \subseteq H$ is immediate, there is some $d_0 \in G$ such that $v(a-d_0) > v(a)$. Then $d_0(0) = a_0$. Suppose that $d_0,\ldots,d_k \in G$ are already constructed such that $d' = d_0+\ldots+d_k$ satisfies $d'(i) = a(i)$ for $i\leq k$. Again, since $G \subseteq H$ is immediate, there is some $d_{k+1} \in G$ such that $v((a-d')-d_{k+1}) > v(a-d') \geq k+1$. Thus, $d_{k+1}(i) = (a-d')(i) = 0$ for $i\leq k$ and $d_{k+1}(k+1) = (a-d')(k+1)$. We obtain $(d_0+\ldots + d_{k+1})(i) = a(i)$ for $i\leq k$ and $(d_0+\ldots + d_{k+1})(k+1) = (d' + d_{k+1})(k+1) = a(k+1)$, as required.
	\end{proof}
	
	\begin{proposition}\thlabel{prop:denseimm}
		Let $G \subseteq H$ be an immediate extension of ordered abelian groups such that $H$ is densely ordered and $vH \subseteq \omega$. Then $G$ is dense in $H$.
	\end{proposition}
	
	\begin{proof}
		Let $\Gamma = vH$. By the Hahn Embedding Theorem (cf. \cite[p.~14]{kuhlmann}), we may consider $H$ as a subgroup of the Hahn product $\H_{n \in \omega} B_n$, where we set $B_n = \{0\}$ for $n \notin \Gamma$. Note that $H \subseteq  \H_{n \in \omega} B_n$ is an immediate extension.
		Let $a,b \in H$ with $0< a < b$. 
		By \thref{lem:denseimm}, there exists a sequence $(d_n)_{n \in \omega}$ in $G$ such that for any $k \in \omega$ and any $i \leq k$ we have $(d_0+\ldots+d_k)(i) = a(i)$. Suppose that $k$ is the last element of $vG$. Then by \thref{prop:denseeq}~(\ref{prop:denseeq:4}), $B_k$ is densely ordered. Let $c \in G$ with $v(c) = k$ and $a(k) < c(k) < b(k)$. Then $d' = d_1+\ldots+d_{k-1} + c \in G$ and $a < d' < b$, as required. Now suppose that $vG$ has no last element. Then let $c \in G$ with $v(c) = m > k$ and $c(m) > a(m)$. Then $d' = d_1+\ldots+d_{m-1} + c \in G$ and $a < d' < b$, as required. 
	\end{proof}
	
	\begin{corollary}\thlabel{cor:denseimm}
		Let $G$ be an ordered abelian group. Suppose that the extension $G \subseteq \div{G}$ is immediate and that $vG \subseteq \omega$. Then $G$ is dense in $\div{G}$.
	\end{corollary}
	
	\begin{proof}
		Apply \thref{prop:denseimm} to $H = \div{G}$.
	\end{proof}
	
	Note that withouth the condition $vG \subseteq \omega$ in \thref{cor:denseimm}, the conclusion that $G$ is dense in $\div{G}$ does not hold in general, as the following example will show.
	
	\begin{example}
		Let $H$ be the Hahn product ${\H}_{\gamma \in \omega + 1} \Q$. Let $H'$ be the Hahn sum $H' = \coprod_{\gamma \in \omega + 1} \Q \subseteq H$. Note that $H'\subseteq H$ is an immediate extension (cf. \cite[p.~3]{kuhlmann}). It follows that for any ordered abelian groups $G_1$ and $G_2$ with $H'\subseteq G_1 \subseteq G_2 \subseteq H$, also the extension $G_1\subseteq G_2$ is immediate.
		
		Let $G\subseteq H$ be given by 
		$$
		G = H' + a\Z \text{, where }a = \sum_{\gamma \in \omega} \mathds{1}_\gamma.$$
		Now $\div{G} = H' + a\Q  \subseteq H$, and the extension $G \subseteq \div{G}$ is immediate. Let $c = \frac 1 2 a + \frac 1 3 \mathds{1}_\omega$ and $d = \frac 1 2 a + \frac 2 3 \mathds{1}_\omega$. Then $c,d \in \div{G}$ with $0<c<d$. However, there is no element in $G$ strictly between $c$ and $d$. Thus $G$ is not dense in $\div{G}$.
	\end{example}
	
	\begin{remark}
		Both in \thref{cor:densimpimm} and in \thref{cor:denseimm}, we imposed a condition on $vG$. Moreover, we provided counterexamples when this condition is not satified. In conclusion, there is exactly one case in which the condition on the value set of $G$ in \thref{cor:densimpimm} and \thref{cor:denseimm} are both satisfied, namely when $vG \cong \omega$. We obtain the following: Let $G$ be an ordered abelian group and $vG \cong \omega$. Then the extension $G\subseteq \div{G}$ is dense if and only if it is immediate.
	\end{remark}

	The final results of this section will be used in \Autoref{sec:lrdef} to compare $\Lr$- and $\Lor$-definability of henselian valuations with real closed residue field in almost real closed fields.
	
	\begin{proposition}\thlabel{prop:convdivlimit}
		Let $G$ be an ordered abelian group. Suppose that for every prime $p \in \N$, there is a $p$-divisible convex subgroup $G_p\neq \{0\}$ of $G$. Then $G$ is closed in $\div{G}$.
	\end{proposition}

	\begin{proof}
		If $G$ is divisible, then the conclusion is trivial. Otherwise, let $a \in G$ and $N \in \N$ such that $\frac a N \in \div{G}\setminus G$. Let $N = p_1\ldots p_m$ be the prime factorisation of $N$. Consider the non-trivial $N$-divisible convex subgroup $H = \bigcap_{i=1}^m G_{p_i}$. Let $h \in H$ with $h>0$. Consider the interval $I=\brackets{\frac aN - h, \frac aN + h}$ in $\div{G}$. Assume that $ I \cap G \neq \emptyset$. Then for any $g \in I \cap G$ we have $v(Ng-a) = v\brackets{g-\frac aN} \geq v(h)$ and thus $Ng- a \in H$. Since $H$ is $N$-divisible, we obtain $g - \frac aN \in H$ and thus $\frac aN\in G$, a contradiction. This gives us  $I \cap G = \emptyset$, as required.
	\end{proof}
	
	\begin{corollary}\thlabel{cor:convdivlimit}
		Let $G$ be an ordered abelian group. Suppose that $G$ has a convex divisible subgroup $H \neq \{0\}$. Then $G$ is closed in $\div{G}$.
	\end{corollary}

	\begin{proof}
		For any prime $p$, let $G_p = H$. The conclusion follows from \thref{prop:convdivlimit}.
	\end{proof}
	
	The following example will show that the converse of \thref{prop:convdivlimit} does not hold.
	
	\begin{example}\thlabel{ex:conversedivclosed}
		Let $2=p_0<p_1<p_2<\ldots$ be a complete list of all prime numbers. For $n \in \N$, let $A_n$ be the following ordered abelian group
		$$A_n = \setbr{\left. \frac{a}{p_1^{m_1}\ldots p_n^{m_n}}  \ \right| a,m_1,\ldots,m_n \in \Z }.$$
		Then $A_n$ is $p_i$-divisible for $i=1,\ldots,n$. 
		Let $G = \coprod_{n \in \N} A_n$. Note that $\div{G} = \coprod_{n \in \N} \Q$. Since $A_n$ is not $2$-divisible for any $n \in \N$, there is no non-trivial $2$-divisble convex subgroup of $G$. Moreover, for any prime $p_i \neq 2$, the maximal $p_i$-divisible subgroup of $G$ is $ \coprod_{n \in \N_{\geq i}} A_n$.
		
		Let $s \in \div{G} \setminus G$ and $N = \min\setbr{ n \in \N \mid s(n) \notin A_n }$. Then for any $a \in \Q$ with $a> 0$ there is no element in $G$ between $s-a\mathds{1}_{N+1}$ and $s+a\mathds {1}_{N+1}$.
		Hence, $s$ is not a limit point of $G$ in $\div{G}$. Since $s$ was arbitrary, we obtain $\cl(G) = G$.
	\end{example}

	\section{Definable Convex Valuations}\label{sec:defconvval}
	
	In this section, we firstly investigate what covex valuations are $\Lor$-de\-fin\-a\-ble in ordered fields and secondly we compare our results to known $\Lr$-definability results of henselian valuations with special focus on almost real closed fields.

	\subsection{$\Lor$-definability}\label{subsec:def}
	
	We are going to analyse the construction method of $\Lor$-definable convex valuations from \cite[Proposition~6.5]{jahnke}.
	
	\begin{fact}\thlabel{prop:defval}
		Let $(K,<)$ be an ordered field. Then at least one of the following holds.
		\begin{enumerate}[wide, labelwidth=!, labelindent=6pt]
			\item $K$ is dense in its real closure.
			\item $K$ admits a non-trivial $\L_{\mathrm{or}}$-definable convex\footnote{\cite[Proposition~6.5]{jahnke} only states that $K$ admits a non-trivial $\Lor$-definable valuation, but the proof indeed gives a construction method for a non-trivial $\Lor$-definable \emph{convex} valuation.} valuation.
		\end{enumerate}
	\end{fact}

	We will summarise the construction procedure of a non-trivial $\L_{\mathrm{or}}$-definable convex valuation ring of an ordered field which is not dense in its real closure given in \cite[p.~163~f.]{jahnke}. For an ordered field $(K,<)$, we denote its real closure by $\rc{K}$ and its topological closure in $\rc{K}$ under the order topology by $\cl(K)$.
	
	\begin{construction}\thlabel{constr:val}
		Let $(K,<)$ be an ordered field. Suppose that $K$ is not dense in $R=\rc{K}$. Let $s\in R\setminus \cl(K)$. Set $D_s:=\{z\in K \mid z < s\}$ and $A_s := \{x \in K^{\geq0} \mid x+D_s\subseteq D_s\}$. Set $\OO_s := \{x \in K \mid |x|A_s \subseteq A_s\}$. Then $\OO_s$ is a non-trivial $\Lor$-definable convex valuation ring of $K$.
	\end{construction}
	
	The following proposition shows that for an ordered Hahn field $K=k\pow{G}$, we can apply \thref{constr:val} to any $s\in R\setminus K$.
	
	\begin{proposition}\thlabel{prop:hahnclosed}
		Let $(k,<)$ be an ordered field and $G\neq \{0\}$ an ordered abelian group. Then $\cl\brackets{k\pow{G}} = k\pow{G}$.
	\end{proposition}
	
	\begin{proof}
		Let $K = k\pow{G}$ and $ R = \rc{K}$. We need to show that $R\setminus K$ is open in $R$. If $K$ is real closed, then $R\setminus K = \emptyset$. Hence, assume that $K$ is not real closed.
		Let $s \in R \setminus K$. Then $s$ is of the form $s = s_1 + at^{g_0} + s_2$ for some $s_1 \in k\pow{G^{<g_0}}$, $s_2 \in \rc{k}\pow{\brackets{\div{G}}^{>g_0}}$ and $at^{g_0} \notin K$. In other words, $at^{g_0}$ is the monomial of $s$ of least exponent which is not contained in $K$. Let $g_1 \in G^{>g_0}$. Then the open interval
		$$I=(s-t^{g_1},s+t^{g_1}) \subseteq R$$
		contains $s$. However, any element in $I$ contains a monomial of the form $at^{g_0}$ and is thus not contained in $K$. Hence, $s$ is contained in an open neighbourhood in $R \setminus K$, as required.
	\end{proof}
	
	\begin{remark}\thlabel{rmk:denseimmediatefield}
		Let $(k,<)$ be an ordered field and $G\neq \{0\}$ an ordered abelian group. \thref{prop:hahnclosed} shows, in particular, that if $k\pow{G}$ is not real closed, then it is not dense in its real closure.
		Note that this can also be shown as follows: 
		Suppose that $k\pow{G}$ is dense in $\rc{k\pow{G}}$. \cite[Lemma~1.31]{kuhlmann} shows that any dense extension of ordered field is immedate. Thus, $k\pow{G} \subseteq \rc{k\pow{G}}$ is an immediate extension with respect to the valuation $\vmin$. Hence, $k= \rc{k}$ and $G = \div{G}$. This implies that $k\pow{G}$ is real closed.
	\end{remark}
	
	Recall the notion of limit points for ordered abelian groups from \Autoref{sec:densdivhull}. We use a similar notion for (left-, right-sided) limit points of an extension of ordered fields.
	
	\begin{theorem}\thlabel{lem:defvalnew}\thlabel{thm:defval}
		Let $(L,<)$ be an ordered field and $w$ a henselian valuation on $L$. Suppose that at least one of the following holds.
		
		\begin{enumerate}[wide, labelwidth=!, labelindent=6pt]
			\item\label{lem:defvalnew:1} $wL$ is discretely ordered.
			
			\item\label{lem:defvalnew:2} $wL$ has a limit point in $\div{wL}\setminus wL$.
			
			\item\label{lem:defvalnew:3} $Lw$ has a limit point in $\rc{Lw}\setminus Lw$.
		\end{enumerate}
		Then $w$ is $\Lor$-definable in $L$. Moreover, in the cases (\ref{lem:defvalnew:1}) and (\ref{lem:defvalnew:2}), $w$ is definable by an $\Lor$-formula with one parameter.
	\end{theorem}
	
	\begin{proof}
		Let $k=Lw$ and $G=wL$. 
		Note that $L$ is not real closed, as in each case either $k$ is not real closed or $G$ is non-divisible. 
		By the Ax--Kochen--Ershov Principle, $(L,<, v) \equiv (k\pow{G},<, \vmin)$. Let $K = k\pow{G}$, $v = \vmin$ and $R = \rc{K}$.
		We will apply \thref{constr:val} with some simplifications to define the valuation ring $k\pow{G^{\geq 0}}$ of $v$ in $K$. \thref{prop:hahnclosed} shows that we can apply the construction procedure to any element in $s \in  R \setminus K$.
		
		First suppose that $G$ is non-divisible. Let $g_0 \in \div{G}\setminus G$ and $s = t^{g_0}$.
		Consider the $\Lor$-definable set $D_s' = \setbr{x \in K^{\geq 0} \mid x < t^{g_0}}$. Since $g_0 \in \div{G}$, there is some $h \in G$ and $N \in \N$ such that $g_0 = \frac{h}{N}$. Thus, the set $D_s'$ is defined by the $\Lor$-formula with one parameter $$x \geq 0 \wedge x^N < t^h.$$ Note that for any $x \in K^{\geq0}$, we have $x \in D_s'$ if and only if $v(x) > g_0$. Thus, $D_s' = k\pow{G^{>g_0}}^{\geq0}$. 
		Let $\OO_s = \{x \in K \mid |x|D_s' \subseteq D_s'\}$. Note that this set is $\Lor$-definable with one parameter. By definition, $\OO_s$ contains exactly those elements in $K$ such that for any $y\in K^{\geq 0}$ with $v(y) > g_0$ we have \begin{align}v(x)+v(y) = v(xy) > g_0. \label{eq:cond1}\end{align} In particular, for any $x \in K$ with $v(x) \geq 0$, condition \eqref{eq:cond1} holds. Thus, $k\pow{G^{\geq0}} \subseteq \OO_s$. To show the other set inclusion, we will make a case distinction, also specifying the element $g_0$ for the densely ordered case.
		
		Suppose that $G$ is discretely ordered. Let $g_1 \in G$ be the least element greater than $g_0$ and let $g_2 \in G$ be the least element greater than $g_0-g_1$. Then $g_2+g_1$ is the least element greater than $g_0$. By choice of $g_1$, this gives us $g_2+g_1 = g_1$ and thus $g_2 = 0$.
		Let $x \in \OO_s$. Since $t^{g_1} \in D_s'$, we have $v(xt^{g_1}) = v(x) + g_1 > g_0$. Hence, $v(x) > g_0 - g_1$. By choice of $g_2$ as the least element greater than $g_0 - g_1$, we obtain $v(x) \geq g_2 = 0$. This implies $\OO_s \subseteq k\pow{G^{\geq0}}$, as required.
		
		Suppose that $G$ has a limit point in $\div{G}\setminus G$. In this case, we choose $g_0 \in \div{G}\setminus G$ such that $g_0$ is a limit point of $G$. We may assume that $g_0$ is a right-sided limit point, as otherwise we can replace it by $-g_0$. Let $x \in K \setminus k\pow{G^{\geq0}}$, i.e. $v(x) < 0$. Since $g_0$ is a right-sided limit point of $G$ in $\div{G}$, the interval $(g_0, g_0 - v(x))\subseteq \div{G}$ contains some element $g_1 \in G$. Thus, $g_1 > g_0$ but $v(x) + v(t^{g_1}) = v(x) + g_1 < g_0$. This shows that $x$ does not fulfil condition \eqref{eq:cond1}, whence $x \notin \OO_s$. We thus obtain $\OO_s \subseteq k\pow{G^{\geq0}}$.
		
		Now suppose that $k$ is not real closed and has a limit point $a$ in $\rc{k}\setminus k$. We may assume that $a$ is a left-sided limit point, as otherwise we can replace it by $-a$. Then $D_a' = \{x \in K \mid a- 1 < x < a\}$ consists exactly of the elements of the form $b + r$, where $b\in k$ such that $ a-1 < b < a$ and $r \in k\pow{G^{>0}}$. In other words, $D_a' = I + k\pow{G^{>0}}$, where $I$ is the convex set $(a-1,a)$ in $k$. Note that $I$ is non-empty, as $a$ is a left-sided limit point of $k$. Let $A_a'$ be the $\Lor$-definable set $\setbr{x \in K^{\geq 0} \mid x + D_a' \subseteq D_a'}$. Since $k\pow{G^{>0}}$ is closed under addition, we have $k\pow{G^{>0}} + D_a' \subseteq D_a'$. Thus, $k\pow{G^{>0}}^{\geq 0} \subseteq A_a'$. For the other inclusion, let $x \in K^{\geq0} \setminus k\pow{G^{>0}}$, i.e. $v(x) \leq 0$ and $x \geq 0$. If $v(x) > 0$, then $x + b \notin D_a'$ for any $b \in I$. Thus, $x \notin A_a'$. Suppose that $v(x) = 0$. Then $x$ is of the form $c + r$ with $c \in k^{>0}$ and $r \in k\pow{G^{>0}}$. If $c \geq 1$, then $x + b \notin D_a'$ for any $b \in I$, whence $x \notin A_a'$. If $c < 1$, let $b \in k\cap (a-c,a)$, which exists, as $a$ is a left-sided limit point of $k$. Then $x + b = (c+b) + r > a + r$. Thus, $x + b \notin D_a'$ and $x \notin A_a'$. 
		Hence, we have shown that $A_a' \subseteq k\pow{G^{>0}}^{\geq0}$. 
		
		Now $(-A_a' \cup A_a') = k\pow{G^{>0}}$ is the maximal ideal of the valuation ring $k\pow{G^{\geq 0}}$. Thus, the valuation ring $k\pow{G^{\geq 0}} = \{x \in K \mid x(-A_a' \cup A_a') \subseteq (-A_a' \cup A_a')\}$ is $\Lor$-definable.
		
		Now for any of the three cases, there is an $\Lor$-formula $\varphi(x,\ul{y})$ (where in the cases (\ref{lem:defvalnew:1}) and (\ref{lem:defvalnew:2}) $\ul{y}$ is just one free variable) such that $(K,<,v) \models \exists \ul{y}\forall x \ (\varphi(x,\ul{y}) \leftrightarrow v(x) \geq 0)$. By elementary equivalence, there is some $\ul{b} \in L$ such that $(L,<,w) \models \forall x \ (\varphi(x,\ul{b}) \leftrightarrow w(x) \geq 0)$. In other words, $\varphi(x,\ul{b})$ defines $w$ in $L$, as required.
	\end{proof}

	\thref{lem:defvalnew} is not a full characterisation of all $\Lor$-definable henselian valuations on an ordered field. Indeed, we can choose $K = \R\pow{G}$ for $G$ as in \thref{ex:conversedivclosed}. Then $\vmin = v_0$ (see p.~\pageref{def:v0}) satisfies neither of the conditions of \thref{lem:defvalnew}, but $v_0$ is even $\Lr$-definable by \thref{fact:thm44} for $p=2$.

	\begin{corollary}\thlabel{cor:kdensedef}
		Let $(L,<)$ be an ordered field and $w$ a henselian valuation on $L$. Suppose that $Lw$ is not real closed and dense in $\rc{Lw}$.
		Then $w$ is $\Lor$-definable in $L$.
	\end{corollary}

	\begin{proof}
		This follows immediately from \thref{lem:defvalnew}~(\ref{lem:defvalnew:3}), as any point in $\rc{Lw}\setminus Lw$ is a limit point of $Lw$ in $\rc{Lw}$.
	\end{proof}
	
	\begin{remark}\thlabel{rmk:defthm}
		\begin{enumerate}[wide, labelwidth=!, labelindent=6pt]
			\item\label{rmk:defthm:1} Let $(k,<)$ be an archimedean ordered field. Then $\Q \subseteq k \subseteq \rc{k} \subseteq \R$. Since $\Q$ is dense in $\R$, also $k$ is dense in $\rc{k}$. In other words, any archimedean ordered field is dense in its real closure. Thus, a special case of \thref{cor:kdensedef} is the following: Let $(L,<)$ be a non-archimedean ordered field and let $w$ be a henselian valuation on $L$ such that $(Lw,<)$ is archimedean and not real closed. Then $w$ is $\Lor$-definable in $L$.
			
			\item \thref{lem:defvalnew}~(\ref{lem:defvalnew:2}) implies a similar version of \thref{cor:kdensedef} if $wL$ is non-divisible and dense in $\div{wL}$. However, we will see in \Autoref{sec:lrdef} that under this condition we already have that $w$ is $\Lr$-definability without parameters.
		\end{enumerate}
	\end{remark}
	
	\noindent {\bf Comparison to $\Lr$-definability.}\label{sec:lrdef} 
	There is a vast collection of results giving conditions on $\Lr$-definability of henselian valuations in pure fields, many of which are from recent years (see e.g. \cite{delon, hong, jahnke4, prestel, jahnke5}). A survey on $\Lr$-definability of henselian valuations is given in \cite{fehm}. We will give a brief account of the known $\Lr$-definability result of henselian valuations in the case that the value group is regular (cf. \cite[Theorem~4]{hong}) and compare this to \thref{lem:defvalnew}.
	
	\begin{fact}\emph{(See \cite[Theorem~4]{hong}.)}\thlabel{fact:hong}
		Let $K$ be a field and $v$ a henselian valuation on $K$. Suppose that $vK$ is regular and non-divisible. Then $v$ is parameter-free $\Lr$-definable in $K$.
	\end{fact}

	\thref{ex:regnew}~(\ref{ex:regnew:1}) shows that there are discretely ordered abelian groups which are not regular. \thref{ex:regnew}~(\ref{ex:regnew:2}) exhibits a densely ordered abelian group $G$ which is not regular but has limit points in $\div{G} \setminus G$. This shows that there are ordered fields such that the cases (\ref{lem:defvalnew:1}) and (\ref{lem:defvalnew:2}) of \thref{lem:defvalnew} are not already covered by \thref{fact:hong}. 
	
	\subsection{Almost Real Closed Fields}\label{sec:arc}
	
	Algebraic and model theoretic properties of the class of almost real closed fields have been studied in \cite{delon}. Moreover, \cite[Theorem~4.4]{delon} gives a complete characterisation of $\Lr$-definable henselian valuations in almost real closed fields. In the following, we will firstly prove some useful properties of almost real closed fields in the language $\Lor$ and secondly compare $\Lr$- and $\Lor$-definability of henselian valuations in almost real closed fields.
	
	\begin{definition} \thlabel{def:arc}
		Let $(K,<)$ be an ordered field, $G$ an ordered abelian group and $v$ a henselian valuation on $K$. We call $K$ an \textbf{almost real closed field (with respect to $v$ and $G$)} if $Kv$ is real closed and $vK = G$. 
	\end{definition}

	Depending on the context, we may simply say that $(K,<)$ is an almost real closed field without specifying the henselian valuation $v$ or the ordered abelian group $G = vK$. 

	\begin{remark}
		In \cite{delon}, almost real closed fields are defined as pure fields which admit a henselian valuation with real closed residue field. However, any such field admits an ordering, which is due to the Baer--Krull Representation Theorem (cf. \cite[p.~37~f.]{engler}). We consider almost real closed fields as ordered fields with a fixed order.
	\end{remark}
	
	Due to \thref{fact:hensconv} and the following fact, we do not need to make a distinction between convex and henselian valuations in almost real closed fields.
	
	\begin{fact}\emph{(See \cite[Proposition~2.9]{delon}.)}\thlabel{fact:convhens}
		Let $(K,<)$ be an almost real closed field. Then any convex valuation on $(K,<)$ is henselian.
	\end{fact}
	
	We will start by showing that several model theoretic results from \cite{delon} in the language $\Lr$ also apply to almost real closed fields in the language $\Lor$.
	\cite[Proposition~2.8]{delon} implies that the class of almost real closed fields in the language $\Lr$ is closed under elementary equivalence. We can easily deduce that this also holds in the language $\Lor$.
	
	\begin{proposition}\thlabel{prop:arcelem}
		Let $(K,<)$ be an almost real closed field and let $(L,\linebreak <) \equiv (K,<)$. Then $(L,<)$ is an almost real closed field.
	\end{proposition}
	
	\begin{proof}
		Since $L \equiv K$, we obtain by \cite[Proposition~2.8]{delon} that $L$ admits a henselian valuation $v$ such that $Lv$ is real closed. Hence, $(L,<)$ is almost real closed.
	\end{proof}
	
	\begin{corollary}\thlabel{cor:arcarc}
		Let $(K,<)$ be an ordered field. Then $(K,<)$ is almost real closed if and only if $(K,<) \equiv (\R\pow{G},<)$ for some ordered abelian group $G$. 
	\end{corollary}
	
	\begin{proof}
		The forward direction follows from the Ax--Kochen--Ershov Principle. The backward direction is a consequence of \thref{prop:arcelem}.
	\end{proof}

	\begin{corollary}\thlabel{prop:arcarc}
		Let $(k,<)$ be an almost real closed field and $G$ an ordered abelian group. Then $(k\pow{G},<)$ is almost real closed.
	\end{corollary}
	
	\begin{proof}
		Let $v$ be a henselian valuation on $k$ such that $kv$ is real closed. By the Ax--Kochen--Ershov principle, we have $(k\pow{G},<,\vmin) \equiv (\R\pow{vk}\pow{G},<,\linebreak \vmin)$.
		Now $(\R\pow{vk}\pow{G},<) \cong (\R\pow{G \oplus vk},<)$, which is an almost real closed field. By \thref{cor:arcarc}, $(k\pow{G},<)$ is almost real closed.
	\end{proof}
	
	Let $(K,<)$ be an almost real closed field. We denote by $V(K)$ the set of all henselian valuations on $K$ with real closed residue field, by $v_1$ the maximum of $V(K)$, i.e. the finest valuation in $V(K)$, and by $v_0$\label{def:v0} be the minimum of $V(K)$, i.e. the coarsest valuation in $V(K)$. These exist by \cite[Proposition~2.1]{delon}.
	By the remarks in \cite[p.~1147~f.]{delon}, $v_0$ is the only possible $\Lr$-definable henselian valuation in $V(K)$. Also in the language $\Lor$, there is at most one definable valuation in $V(K)$.

	\begin{proposition}\thlabel{prop:rcuniquedef}
		Let $(K,<)$ be an almost real closed field and $v \in V(K)$. Suppose that $v$ is $\Lor$-definable in $K$. Then $v$ is the only $\Lor$-definable valuation in $V(K)$. 
	\end{proposition}	
	
	\begin{proof}
		By the Ax--Kochen--Ershov Principle, we have $$(K,<,v) \equiv (\R\pow{vK},<,\vmin).$$ Since $v$ is $\Lor$-definable in $K$, there exists an $\Lor$-formula $\varphi(x,\ul{y})$ such that $$K \models \exists \ul{y} \forall x\  (\varphi(x,\ul{y}) \leftrightarrow v(x) \geq 0).$$ By elementary equivalence, there exists $\ul{b} \in \R\pow{vK}$ such that $$\R\pow{vK} \models \forall x\  (\varphi(x,\ul{b}) \leftrightarrow \vmin(x) \geq 0).$$ Hence, $\vmin$ is $\Lor$-definable in $\R\pow{vK}$. 
		
		Let $\psi(x,\ul{y})$ be an $\Lor$-formula and $\ul{c} \in K$ such that $\psi(x,\ul{c})$ defines a convex valuation $w$ in $K$. Assume, for a contradiction, that $w$ is strictly finer than $v$, i.e. $\OO_w \subsetneq \OO_v$. This implies
		$$(K,<,v) \models \forall x \ (\psi(x,\ul{c}) \to v(x) \geq 0) \wedge \exists z\ ( \neg \psi(z,\ul{c}) \wedge v(z) \geq 0 ).$$
		By elementary equivalence, there is some $\ul{c}' \in \R\pow{vK}$ such that
		$\psi(x,\ul{c}')$ defines a convex valuation $w'$ in $\R\pow{vK}$ with $\OO_{w'} \subsetneq \OO_{\vmin}$. This contradicts that fact that $\vmin$ is the finest convex valuation on $\R\pow{vK}$. Hence, $v$ is the finest $\Lor$-definable convex valuation in $K$.

		Let $v' \in V(K)$ be $\Lor$-definable. Arguing as above, $v'$ is the finest $\Lor$-definable convex valuation on $K$. This gives us $v' = v$, as required.
	\end{proof}

	\noindent {\bf Comparison of $\Lr$- and $\Lor$-definability of henselian valuations in almost real closed fields.}
	Let $p$ be a prime number. A valuation $v$ on $K$ is called {$p$-Kummer henselian} if Hensel's Lemma holds for polynomials of the form $x^p-a$ for $a \in \OO_v$. A field $L$ is called {$p$-euclidean} if $L = \pm L^p$. Let $V_p(K)$ be the set of all $p$-Kummer henselian valuations of $K$ with $p$-euclidean residue field. Denote by $v_p$ the minimum of $V_p(K)$ (cf. \cite[p.~1126]{delon}).
	
	\begin{fact}\thlabel{fact:thm44}\emph{(See \cite[Theorem~4.4]{delon}.)}
		Let $(K,<)$ be an almost real closed field and $v$ a henselian valuation on $K$. Then $v$ is $\Lr$-definable in $K$ if and only if $vK$ is $\Log$-definable in $v_1K$ and $v \leq v_p$ for some prime $p$. Moreover, $v_0$ is $\Lr$-definable if and only if there is a prime $p$ such that $v_0K$ has no non-trivial $p$-divisible convex subgroups. 
	\end{fact}

	Recall that $v_0$ is the only possible $\Lr$-definable valuation in $V(K)$. If the ordering on an almost real closed field $(K,<)$ is $\Lvf$-definable for $v = v_0 \in V(K)$, then we obtain a complete characterisation of $\Lor$-definable convex valuations in $K$.
	
	\begin{lemma}\thlabel{lem:orderingdef2}
		Let $(K,<)$ be an ordered field and let $v$ be a henselian valuation on $K$ such that $Kv$ is $2$-euclidean (i.e. root-closed for positive elements) and $vK$ is $2$-divisible. Then the ordering $<$ is parameter-free $\Lvf$-definable in $K$.
		In particular, if $v$ is $\Lr$-definable in $K$, then any $\Lor$-definable subset of $K$ is already $\Lr$-definable.
	\end{lemma}
	
	\begin{proof}
		Let $k=Kv$ and $G=vK$. Consider the $\Lvf$-formula $\varphi(x)$ given by $$x = 0 \vee \exists y\ v(x-y^2) > v(x).$$
		We will show that for any $a \in k\pow{G}$, the formula $\varphi(a)$ holds if and only if $a \geq 0$. 
		Let $a = a_gt^g + s \in k\pow{G}^\times$, where $a_g \in k^\times$, $s \in k\pow{G^{>g}}$ and $g = \vmin(a)$.
		Suppose that $\varphi(a)$ holds. Then there exists $y \in K^\times$ such that $\vmin(x-y^2) > g$. Hence, $a_g = y_g^2 > 0$, where $y_g$ is the coefficient of the monomial $t^g$ in $y$. Thus, $a > 0$.
		Now suppose that $a > 0$. Let $y = \sqrt{a_g}t^{g/2}$. Then $\vmin(a-y^2) = \vmin(s) > g = \vmin(a)$.
		
		By the Ax--Kochen--Ershov Principle, $(K,<,v) \equiv (Kv\pow{vK},<, \vmin)$. Hence, we obtain $K \models \forall x \ (x\geq 0 \leftrightarrow \varphi(x))$.
	\end{proof}
	
	\begin{proposition}\thlabel{prop:charlrdef}
		Let $(K,<)$ be an almost real closed field. Suppose that $v_0$ is $\Lr$-definable and that $v_0K$ is $2$-divisible. Let $w$ be a valuation on $K$. If $w$ is $\Lor$-definable, then it is $\Lr$-definable.
	\end{proposition}
	
	\begin{proof}
		Since $v_0K$ is real closed, it is $2$-euclidean. By \thref{lem:orderingdef2}, any $\Lor$-definable valuation on $K$ is already $\Lr$-definable.
	\end{proof}
	
	\begin{remark}
 	We obtain the following characterisation of $\Lor$-definable convex valuations in certain almost real closed fields:
	Let $(K,<)$ be an almost real closed field. Suppose that the value group $v_0K$ is $2$-divisible and, for some prime $p$, it has no non-trivial $p$-divisible convex subgroup. Let $v$ be a convex valuation on $K$. By \thref{fact:convhens}, $v$ is henselian. Thus, by \thref{prop:charlrdef} and \thref{fact:thm44}, $v$ is $\Lor$-definable in $K$ if and only if $vK$ is $\Log$-definable in $v_1K$ and $v \leq v_p$ for some prime $p$.
	\end{remark}

	\thref{prop:charlrdef} shows in particular that for an almost real closed field, if $v_0$ is $\L_r$-definable, then so is any $\Lor$-definable henselian valuation. The question becomes whether there is an almost real closed field which admits a $\Lor$-definable henselian valuation which is not $\Lr$-definable. The final result of this sections shows that any henselian valuation in an almost real closed field satisfying the hypothesis of \thref{thm:defval} is already $\Lr$-definable. Note that any discretely ordered abelian group does not have a non-trivial $n$-divisible convex subgroup for any $n \geq 2$.
	
	\begin{proposition}
		Let $(L,<)$ be an almost real closed field with respect to a henselian valuation $w$ and an ordered abelian group $G$ such that either $G$ is discretely ordered or $G$ is not closed in $\div{G}$. Then $w = v_0$ and $v_0$ is $\Lr$-definable.
	\end{proposition}

	\begin{proof}
		Recall that $v_0$ is the unique henselian valuation on $L$ such that $Lv_0$ is real closed and $v_0L$ has no non-trivial divisible convex subgroup. Moreover, $v_0$ is $\Lr$-definable if and only if for some prime $p$, there is no $p$-divisible non-trivial convex subgroup of $v_0L$. If $wL = G$ is discretely ordered, both conditions are satisfied, and thus $w = v_0$ is $\Lr$-definable. If $G$ is not closed in $\div{G}$, then by \thref{cor:convdivlimit}, $G$ has no non-trivial convex subgroup, whence $w = v_0$, and by \thref{prop:convdivlimit}, there is a prime $p$ such that $G$ has no non-trivial $p$-divisible convex subgroup, whence $v_0$ is $\Lr$-definable.
	\end{proof}


	\section{Density in Real Closure}\label{sec:density}

	In \Autoref{sec:densdivhull} we considered ordered abelian groups which are dense in their divisible hull. Recall that an ordered abelian group is dense in its divisible hull if and only if it is densely ordered and regular. Recall further that the property of density in the divisible hull is preserved under e\-le\-men\-ta\-ry equivalence. In \thref{cor:kdensedef} we started considering the analogue for ordered fields, namely ordered fields which are dense in their real closure. Note that the theory of divisible ordered abelian groups and the theory of real closed fields share several model theoretic properties, such as completeness, o-minimality and quantifier elimination.
	In this section, we will firstly study the class of ordered fields which are dense in their divisible hull model theoretically and secondly give some connections to the literature in which these ordered fields have been studied.
	
	At first, we change to a more general setting of complete o-minimal theories.
	For a structure $\mathcal{M}$ and a subset $A \subseteq M$, denote the definable closure of $A$ in $\mathcal{M}$ by $\dcl(A;\mathcal{M})$.
	
	\begin{theorem}\thlabel{lem:densedc}
		Let $\L$ be a language expanding $\Log$ and let $\MM=(M,+,0,\linebreak <,\ldots)$ and $\NN=(N,+,0,<,\ldots)$ be ordered $\L$-structures such that $(M,+,0,\linebreak <)$ is a non-trivial ordered abelian group and $\MM \equiv \NN$. Suppose that there exists a complete o-minimal $\L$-theory $T\supseteq T_{\mathrm{doag}}$ admitting quantifier elimination such that there are $\MM',\NN' \models T$ with $\MM\subseteq \MM'$, $\NN\subseteq \NN'$, $\dcl(M;\MM')=M'$ and $\dcl(N;\NN')=N'$. Suppose further that $M$ is not dense $M'$. Then $N$ is also not dense in $N'$.
	\end{theorem}
	
	\begin{proof}
		Since $M$ is not dense in $M'$, there exist $\alpha,\beta \in M'$ such that $\alpha < \beta$ and $\alpha$ and $\beta$ produce the same cut on $M$, i.e. \[\setbr{x \in M \mid x < \alpha} = \setbr{x \in M \mid x < \beta}.\] 
		Since $M' = \dcl(M;\MM')$, there are $\L$-formulas $\varphi(\ul{x},y)$ and $\psi(\ul{x},y)$, each defining a $0$-definable function from $(M')^m$ to $M'$ for some $m \in \N$, such that for some $\ul{a} \in M$ we have
		\[\MM' \models \varphi(\ul{a},\alpha) \wedge \psi(\ul{a},\beta).\]
		Let $f$ and $g$ be the functions corresponding to $\varphi$ and $\psi$ respectively. We may assume that for any $\ul{x} \in (M')^m\setminus\{(0,\ldots,0)\}$, we have $f(\ul{x}) \neq g(\ul{x})$, as otherwise we can replace $g$ by the $0$-definable function $g'$ given by $$g'(\ul{x}) = \begin{cases}
		g(\ul{x}), & \text{if } g(\ul{x}) \neq f(\ul{x}),\\
		- g(\ul{x}), & \text{if } g(\ul{x}) = f(\ul{x}) \neq 0,\\
		\max\{|x_1|,\ldots,|x_m|\}, & \text{if } g(\ul{x}) = f(\ul{x}) = 0.
		\end{cases}$$
		
		Now let $\varphi'(\ul{x},z)$ be given by $z < f(\ul{x})$ and let $\psi'(\ul{x},z)$ be given by $ z < g(\ul{x})$.
		By quantifier elimination in $T$, we may take $\varphi''$ and $\psi''$  quantifier-free such that they are equivalent to $\varphi'$ and $\psi'$ respectively.
		Note that for any $b \in M'$ we have $\MM'\models \varphi''(\ul{a},b)$ if and only if $b < \alpha$, and $\MM' \models \psi''(\ul{a},b)$ if and only if $b < \beta$. Hence,
		\[\MM \models  \forall z(\varphi''(\ul{a},z) \leftrightarrow \psi''(\ul{a},z)).\]
		We need to make a case distinction to obtain some element $\ul{a}' \in N$ such that $f(\ul{a}') \neq g(\ul{a}')$.
		
		\textbf{Case (1):} $\ul{a} = 0$. Then \[\MM \models  \forall z(\varphi''(0,z) \leftrightarrow \psi''(0,z)).\] 
		By elementary equivalence, we obtain
		\[\NN \models \forall z(\varphi''(0,z) \leftrightarrow \psi''(0,z)).\]
		Now there are unique $\alpha',\beta' \in N'$ such that
		\[\NN' \models f(0)=\alpha' \wedge g(0)=\beta'.\]
		Since $\MM' \models  f(0) \neq g(0)$, we obtain by completeness of $T$ that $\NN' \models f(0) \neq g(0)$ and thus that $\alpha' \neq \beta'$.
		
		\textbf{Case (2):} $\ul{a} \neq 0$. Then 	\[\MM \models \exists \ul{x} \ (\ul{x} \neq 0 \wedge \forall z(\varphi''(\ul{x},z) \leftrightarrow \psi''(\ul{x},z))).\]
		By elementary equivalence, we obtain
		\[\NN \models \exists \ul{x}\ (\ul{x} \neq 0 \wedge \forall z(\varphi''(\ul{x},z) \leftrightarrow \psi''(\ul{x},z))).\]
		Let $\ul{a}' \in N$ with $\ul{a}' \neq 0$ such that  \[\NN \models \forall z\ (\varphi''(\ul{a}',z) \leftrightarrow \psi''(\ul{a}',z)).\]
		By assumptions on $f$ and $g$, there are unique $\alpha',\beta' \in N'$ such that $\alpha' \neq \beta'$ and
		\[\NN' \models f(\ul{a}')=\alpha' \wedge g(\ul{a}')=\beta'.\]
		This completes the case distinction.
		
		For any $b' \in {N}$ we have $b' < \alpha'$ if and only if $\NN' \models b' < f(\ul{a}')$. By definition of $\varphi'$, this holds if and only if $\NN' \models \varphi'(\ul{a}',b')$. Again, by quantifier elimination, this is equivalent to $\NN' \models \varphi''(\ul{a}',b')$. Since $\varphi''$ is quantifier-free, this holds if and only if $\NN \models \varphi''(\ul{a}',b')$. Similarly, we obtain that for any $b' \in N$ we have $b' < \beta'$ if and only if $\NN \models \psi''(\ul{a}',b')$. 
		
		Hence, we obtain that for any $b' \in N$ we have $b' < \alpha'$ if and only if $b' < \beta'$. This shows that $\alpha'$ and $\beta'$ produce the same cut in $N$, and hence, that $N$ is also not dense in $N'$. 
	\end{proof}
	
	\begin{corollary}\thlabel{thm:densitytransfers}	
		Let $(K,<)$ and $(L,<)$ be ordered fields such that $(K,<) \equiv (L,<)$. Suppose that $K$ is dense in $\rc{K}$. Then $L$ is dense in $\rc{L}$.
	\end{corollary}

	\begin{proof}
		This follows directly from \thref{lem:densedc} applied to $T = \Trcf$, noting that for an ordered field $(K,<)$, the definable closure of $K$ in $\rc{K}$ is $\rc{K}$.
	\end{proof}

	\begin{corollary}\thlabel{thm:densitytransfers2}	
		Let $(K,<)$ be an ordered field which has an ar\-chi\-me\-de\-an model. Then $K$ is dense in $\rc{K}$.
	\end{corollary}
		
	\begin{proof}
		By \thref{rmk:defthm}~(\ref{rmk:defthm:1}), any archimedean ordered field is dense in its real closure. Thus, by \thref{thm:densitytransfers}, $K$ is dense in $\rc{K}$.
	\end{proof}

	\begin{remark}\thlabel{rmk:hahnfieldnoarch}
		Recall that by \thref{prop:hahnnoarch}, any non-divisible Hahn pro\-duct whose rank has no maximal element has no archimedean model. We obtain the following analogue for ordered Hahn fields:
		Let $(k,<)$ be an orderd field and let $G \neq \{0\}$ be an ordered abelian group. Suppose that $k\pow{G}$ is not real closed. By \thref{rmk:denseimmediatefield}, $k\pow{G}$ is not dense in $\rc{k\pow{G}}$, and thus $(k\pow{G},<)$ has no archimedean model.
	\end{remark}

	The class of ordered abelian groups which are dense in their divisible hull is recursively axiomatised by the axiom system for densely ordered regular abelian groups. In the following, we will show that also the class of ordered fields which are dense in their real closure can be recursively axiomatised. We will, again, first consider a model theoretically more general setting.

	\begin{proposition}\thlabel{prop:omindenserec}
		Let $\L$ be a language expanding $\Log$, let $T\supseteq T_{\mathrm{doag}}$ be a complete $\L$-theory which admits quantifier elimination and let $\Sigma \subseteq T$ be a theory extending the theory of ordered abelian groups.
		Then there is a theory $\Sigma' \supseteq \Sigma$ such that for any $\MM' \models T$ and any $\MM \subseteq \MM'$ with $\dcl(M;\MM') = M'$ we have that $\MM \models \Sigma'$ if and only if $\MM \models \Sigma$ and $M$ is dense in $M'$. Moreover, if $\Sigma$ and $T$ are recursive, we can also choose $\Sigma'$ to be recursive.
	\end{proposition}

	\begin{proof}
		For any $0$-definable function $f$ in $T$, 
		let $\varphi_f(\ul{x},z)$ be a quatifier-free formula which is equivalent to $z < f(\ul{x})$ in $T$. Let $A$ be the set of all pairs $(f,g)$ of $0$-definable functions in $T$ with the same arity such that $T \models \forall \ul{x} \ (f(\ul{x}) = g(\ul{x}) \to \ul{x} = 0)$. In other words, $A$ consists of all pairs of $0$-definable functions which are distinct everywhere except possibly in $0$.
		Set 
		$$\Sigma' = \Sigma \cup \setbr{\forall \ul{x}\ (\ul{x} \neq 0 \to \neg \forall z\  (\varphi_f(\ul{x},z) \leftrightarrow \varphi_g(\ul{x},z) )) \mid (f,g) \in A}.$$ Note that if $\Sigma$ and $T$ are recursive, then so is $\Sigma'$, as it is decidable whether a given $\L$-formula defines a function. 
		
		Let $\MM' \models T$ and let $\MM \subseteq \MM'$ with $\dcl(M;\MM') = M'$. Suppose that $\MM \models \Sigma'$. We need to show that $M$ is dense in $M'$. Let $\alpha, \beta \in M'$ such that $\alpha < \beta$. Let $f$ and $g$ be $0$-definable functions and let $\ul{a} \in M$ such that $f(\ul{a}) = \alpha$ and $g(\ul{a}) = \beta$. Since also $\dcl(M\setminus\setbr{0};\MM') = M'$, we may assume that $\ul{a} \neq 0$. Hence, $\MM \models \neg \forall z\  (\varphi_f(\ul{a},z) \leftrightarrow \varphi_g(\ul{a},z) )$. This implies that for some $b \in M$ we either have $\beta \leq b < \alpha$ or $\alpha \leq b < \beta$. Hence, $M$ is dense in $M'$.
		Conversely, suppose that $\MM \models \Sigma$ and $M$ is dense in $M'$. Let $(f,g) \in A$ and $\ul{a} \in M$ with $\ul{a} \neq 0$. We need to show that $\MM \models \neg \forall z\  (\varphi_f(\ul{a},z) \leftrightarrow \varphi_g(\ul{a},z) )$. Let $\alpha = f(\ul{a})$ and $\beta = f(\ul{b})$. By assumption on $A$, we have $\alpha \neq \beta$, say $\alpha < \beta$. Since $M$ is dense in $M'$, there is $c \in M$ such that $\alpha < c < \beta$. Hence, $\MM \models \neg \varphi_f(\ul{a},c)$ but $\MM \models \varphi_g(\ul{a},c)$, as required.
		\end{proof}

	\begin{corollary}\thlabel{cor:denserec}
		There is a recursive $\Lor$-theory $\Sigma'$ which axiomatises the class of ordered fields which are dense in their real closure.
	\end{corollary}

	\begin{proof}
		We apply \thref{prop:omindenserec} to the recursive $\Lor$-theory $T = \Trcf$ and the axiom system $\Sigma$ for ordered fields in order to obtain a recursive $\Lor$-theory $\Sigma'$ with the required properties. 
	\end{proof}

	By \thref{thm:densitytransfers2}, any ordered field with an archimedean model is dense in its divisible hull. We will use \thref{cor:denserec} to show that the converse does not hold, i.e. that there exists an ordered field which is dense in its divisible hull but has no archimedean model.
	
	Let $(K,<)$ be an ordered field. An \textbf{integer part} of $K$ is a discretely ordered subring $(Z,<) \subseteq (K,<)$ with $1$ as least positive element such that for any $x \in K$ there exists $z\in Z$ with $z \leq x < z+1$.
	
	\begin{proposition}\thlabel{prop:ordfieldnoarchmodel}
		There is an ordered field $(K,<)$ such that $K$ is dense in $\rc{K}$ but $(K,<)$ has no archimedean model.
	\end{proposition}
	
	\begin{proof}
		Let $\varphi(x)$ be an $\Lor$-formula defining $\Z$ in $(\Q,<)$ (cf. \cite[Theorem~3.1]{robinson2}). Let $\Sigma_1$ be the $\Lor$-theory $\Sigma'$ from \thref{cor:denserec} and let $\Sigma_2$ be a set of axioms stating that $\varphi(x)$ defines an integer part. For any $\Lor$-formula $\alpha$, let $\widetilde{\alpha}$ be the formula in which all quantifiers $\exists x$ and $\forall x$ are bounded by $\varphi(x)$, that is, all instances of subformulas of the form $\exists x \theta(x,\ul{y})$ are replaced by $\exists x (\varphi(x) \wedge \theta(x,\ul{y}))$ and all instances of subformulas of the form $\forall x \theta(x,\ul{y})$ are replaced by $\forall x (\varphi(x) \to \theta(x,\ul{y}))$. Let $\Sigma$ be the deductive closure of $\Sigma_1 \cup \Sigma_2$. Note that $\Sigma$ is recursive  and $(\Q,<) \models \Sigma$. 
		If for every $\Lor$-sentence $\alpha \in \Th(\Z,+,-,\cdot,0,1,<)$, we had $\Sigma \vdash \widetilde{\alpha}$, then $\Th(\Z,+,-,\cdot,0,1,<)$ would be decidable. Hence, there is an $\Lor$-sentence $\sigma$ such that $\neg \sigma \in \Th(\Z,+,-,\cdot,0,1,<)$ and $\widetilde{\sigma},\neg\widetilde{\sigma} \notin \Sigma$. Let $\widetilde{\Sigma}$ be the consistent $\Lor$-theory $\Sigma \cup \setbr{\widetilde{\sigma}}$ and let $(K,<) \models \widetilde{\Sigma}$. Assume that $(K,<)$ is archimedean. Since $(K,<) \models \Sigma_2$, the formula $\varphi(x)$ defines the integer part $\Z$ of $K$. But then $(K,<) \models \neg \widetilde{\sigma}$, a contradiction. Hence, $(K,<)$ cannot have any archimedean models. However, since $(K,<) \models \Sigma_1$, we have that $K$ is dense in $\rc{K}$.
	\end{proof}

	We now analyse algebraic properties of ordered fields which are dense in their real closure. For an ordered field $K$, let $\mathbf{A}_K$ be a group complement to $\OO_v$ in $K$, i.e. $\mathbf{A}_K$ is an ordered subgroup of $K$ such that $K = \mathbf{A} \oplus \OO_v$. Note that for an extension of ordered field $(K,<) \subseteq (L,<)$, the group complement $\mathbf{A}_K$ can be extended to a group completement $\mathbf{A}_L$. The following is a useful characterisation of dense extension of ordered field (cf. \cite[Lemma~1.32]{kuhlmann}).
	
	\begin{fact}
		Let $(K,<) \subseteq (L,<)$ be an extension of ordered fields. Then $K$ is dense in $L$ if and only if $\mathbf{A}_K = \mathbf{A}_L$.
	\end{fact}
	
	Recall that any dense extension of ordered fields is immediate (cf. \cite[p.~29~f.]{kuhlmann}). In particular, if the extension of ordered fields $(K,<)\subseteq (\rc{K},<)$ is dense, then $Kv$ is real closed and $vK$ is divisible. The converse holds under the additional assumption that $vK$ is archimedean (cf. \cite[Lemma~19]{kuhlmann2}).
	
	\begin{fact}\thlabel{fact:immimpdense}
		Let $(K,<)$ be an ordered field such that $Kv$ is real closed and $vK$ is divisible and archimedean. Then $K$ is dense in $\rc{K}$.
	\end{fact}

	\cite[Remark~3.5]{biljakovic}
	provides an example of an ordered field field $(K,<)$ such that $Kv$ is real closed and $vK$ can be chosen to be divisible and non-archimedean (i.e. $K\subseteq \rc{K}$ is immediate) but $\mathbf{A}_K = \mathbf{A}_{\rc{K}}$, whence $K$ is not dense in $\rc{K}$.
	
	We conclude this section by analysing density in real closure in terms of dense transcendence bases. For a field $F$, we denote the transcendence degree over its prime field by $\td(F)$. We say that $T \subseteq F$ is a transcendence basis of $F$ if it is a transcendence basis over its prime field.
	
	The following result is due to \cite[Lemma~2.3]{erdos}.
	
	\begin{fact}\thlabel{lem:lemmaerdos}
		Let $(K,<)$ be an uncountable ordered field. Then there exists a transcendence basis $T$ of $K$ over $\Q$ such that $T$ is dense in $K$.
	\end{fact}

	The arguments in \cite{erdos} can be generalised to countable ordered fields with countably infinite transcendence basis.
	
	\begin{proposition}\thlabel{prop:lemmaerdos}
		Let $(K,<)$ be an ordered field with $\td(K) = \aleph_0$. Then there exists a transcendence basis $T=\{t_1,t_2,\ldots\}$ of $K$ such that $T$ is dense in $K$.
	\end{proposition}
	
	\begin{proof}
		Let $(I_n)_{n\in\N}$ be an enumeration of all intervals $(a,b) \subseteq K$. We construct a transcendence basis $\{t_1,t_2,\ldots\}$ of $K$ over $\Q$ such that $t_k \in I_k$ for any $k$.
		
		Let $t_1 \in I_1$ be an arbitrary element transcendental over $\Q$. Suppose that $t_1,\ldots,t_{n}$ have already been chosen for some $n$. If all elements in $I_{n+1}$ were algebraic over $\Q(t_1,\ldots,t_n)$, then also $K$ would be algebraic over $\Q(t_1,\ldots,t_n)$, contradicting its transcendence degree. Hence, we can choose $t_{n+1} \in I_{n+1}$ which is transcendental over $\Q(t_1,\ldots,t_n)$.
	\end{proof}

	\begin{corollary}\thlabel{cor:lemmaerdos}
		Let $(K,<)$ be an ordered field. Suppose that $\td(K) \geq \aleph_0$. Then $K$ is dense in $\rc{K}$ if and only if $K$ has a transcendence basis $T$ which is dense in $\rc{K}$.
	\end{corollary}

	\begin{proof}
		Note that any transcendence basis of $K$ is a transcendence basis of $\rc{K}$. If $K$ has a transcendence basis $T$ which is dense in $\rc{K}$, then $K$ has a proper subset dense in $\rc{K}$ and thus $K$ itself is dense in $\rc{K}$. Conversely, suppose that $K$ is dense in $\rc{K}$. By \thref{lem:lemmaerdos} and  \thref{prop:lemmaerdos}, $K$ has a transcendence $T$ basis which is dense in $K$. Since $K$ is dense in $\rc{K}$, also $T$ is dense in $\rc{K}$.
	\end{proof}

	\section{Strongly NIP Ordered Fields}\label{sec:conjecturalclassification}

	In this section we will study the class of strongly NIP ordered fields in the light of \thref{conj:main} and \thref{conj:classification}. There are several results known about strongly NIP fields in the language $\Lr$ (cf. \cite{halevi2, halevi}). Note that throughout this section, we consider ordered fields which are strongly NIP in the language $\Lor$. 
	
	A special class of strongly NIP ordered fields are dp-minimal ordered fields. These are fully classified in \cite{jahnke}. In  \thref{prop:dparc} below we show that our query \ref{query} holds for dp-minimal ordered fields.
	An ordered group $G$ is called \textbf{non-singular} if $G/pG$ is finite for all prime numbers $p$.
	
	\begin{fact}\thlabel{fact:dpgroup} \emph{(See \cite[Proposition~5.1]{jahnke}.)}
		An $\aleph_1$-saturated ordered abelian group $G$ is dp-minimal if and only if it is non-singular.
	\end{fact}
	
	\begin{example}\thlabel{ex:dpmin}
		Let $G$ be an $\aleph_1$-saturated extension of the ordered abelian group $\Z$. Since $\Z$ is non-singular, by \thref{fact:dpgroup} $G$ is a dp-minimal ordered abelian group which is discretely ordered.
	\end{example}
	
	\begin{fact}\thlabel{fact:dpfield} \emph{(See \cite[Theorem~6.2]{jahnke}.)}
		An ordered field $(K,<)$ is dp-minimal if and only if there exists a non-singular ordered abelian group $G$ such that $(K,<) \equiv (\R\pow{G},<)$. 
	\end{fact}
	
	\begin{lemma}\thlabel{lem:dpminresidue}
		Let $(K,<)$ be a dp-minimal almost real closed field with respect to some henselian valuation $v$. Then $vK$ is dp-minimal.
	\end{lemma}
	
	\begin{proof}
		Since $Kv$ is an ordered field, it is not separably closed. Thus, by \cite[Theorem~A]{jahnke2}, $v$ is definable in the Shelah expansion $(K,<)^{\mathrm{Sh}}$ (cf. \cite[Section~2]{jahnke2}) of $(K,<)$. By \cite[Observation~3.8]{onshuus}, also $(K,<)^{\mathrm{Sh}}$ is dp-minimal, whence the reduct $(K,v)$ is dp-minimal. Since $Kv$ is real closed, $Kv^\times/\linebreak (Kv^\times)^n$ is finite for all $n \in \N$. Hence, by \cite[Proposition~6.1]{jahnke} also $vK$ is dp-minimal.
	\end{proof}
	
	\begin{proposition}\thlabel{prop:dparc}\thlabel{prop:dpminarc}
		Let $(K,<)$ be an ordered field. Then $(K,<)$ is dp-minimal if and only if it is almost real closed with respect to a dp-minimal ordered abelian group.
	\end{proposition}

	\begin{proof}
		Suppose that $(K,<)$ is almost real closed with respect to a dp-minimal ordered abelian group $G$. By \thref{fact:dpgroup}, an $\aleph_1$-saturated elementary extension $G_1$ of $G$ is non-singular. By the Ax--Kochen--Ershov Principle, we have $(K,<) \equiv (\R\pow{G_1},<)$, which is dp-minimal by \thref{fact:dpfield}. Hence, $(K,<)$ is dp-minimal.
		
		Conversely, suppose that $(K,<)$ is dp-minimal. By \thref{fact:dpfield}, $(K,<) \equiv (\R\pow{G},<)$ for some non-singular ordered abelian group $G$. Since $(\R\pow{G},\linebreak<)$ is almost real closed, by \thref{prop:arcelem} also $(K,<)$ is almost real closed with respect to some henselian valuation $v$. By \thref{lem:dpminresidue}, also $vK$ is dp-minimal, as required.
	\end{proof}

	As a result, we obtain a characterisation of dp-minimal archimedean ordered fields.
	
	\begin{corollary}\thlabel{prop:dparch}
		Let $(K,<)$ be a dp-minimal archimedean field. Then $K$ is real closed.
	\end{corollary}
	
	\begin{proof}
		The only archimedean almost real closed fields are the archimedean real closed fields. Thus, by \thref{prop:dpminarc}, any archimedean dp-minimal ordered field is real closed.
	\end{proof}

	
	 We now turn to strongly NIP almost real closed fields. The first result will be used in the proof of \thref{thm:arcf} for the classification of strongly NIP almost real closed fields.
	 
	 \begin{proposition}\thlabel{lem:strongnipfieldresidue}
	 	Let $(K,<)$ be a strongly NIP ordered field and let $v$ be a henselian valuation on $K$. Then also $(Kv,<)$ and $vK$ are strongly NIP.
	 \end{proposition}
	 
	 \begin{proof}
	 	Arguing as in the proof of \thref{lem:dpminresidue}, we obtain that $v$ is definable in $(K,<)^{\mathrm{Sh}}$. Now $(K,<)^{\mathrm{Sh}}$ is also strongly NIP (cf. \cite[Observation~3.8]{onshuus}), whence $(K,<,v)$ is strongly NIP. By \cite[Observation~1.4~(2)]{shelah}, any structure which is first-order intepretable in $(K,<,v)$ is strongly NIP. Hence, also $(Kv,<)$ and $vK$ are strongly NIP.
	 \end{proof}

	 \thref{conj2impconj1} and \thref{prop:conj1toconj2} below are used in the proof of \thref{thm:main}.
	 For the first result, we adapt \cite[Lemma~1.9]{halevi} to the context of ordered fields.

	 \begin{proposition}\thlabel{conj2impconj1}
	 	Assume that any strongly NIP ordered field is either real closed or admits a non-trivial henselian valuation. 	 Let $(K,<)$ be a strongly NIP ordered field. Then $(K,<)$ is almost real closed with respect to the canonical valuation, i.e. the finest henselian valuation on $K$.
	 \end{proposition}
	 
	 \begin{proof}
	 	Let $(K,<)$ be a strongly NIP ordered field. If $K$ is real closed, we can take the natural valuation. Otherwise, by assumption, the set of non-trivial henselian valuations on $K$ is non-empty. Let $v$ be the canonical valuation on $K$.
	 	By \thref{lem:strongnipfieldresidue}, $(Kv,<)$ is strongly NIP. Note that $Kv$ cannot admit a non-trivial henselian valuation, as otherwise this would induce a non-trivial henselian valuation on $K$ finer than $v$. Hence, by assumption, $Kv$ must be real closed.
	 \end{proof}
	 
	 The next result is obtained from a slight adjustment of the proof of \cite[Fact~1.8]{halevi}.
	
	\begin{proposition}\thlabel{prop:conj1toconj2}
		Let  $(K,<)$ be a strongly NIP ordered field which not real closed but is almost real closed with respect to a henselian valuation $v$. Then there exists a non-trivial $\Lr$-definable henselian coarsening of $v$.
	\end{proposition}
	
	\begin{proof}
	  By \thref{lem:strongnipfieldresidue}, $vK=G$ is strongly NIP. Since $K$ is not real closed, $G$ is non-divisible. By \cite[Proposition~5.5]{halevi2}, any henselian valuation with non-divisible value group on a strongly NIP field has a non-trivial $\Lr$-definable henselian coarsening. Hence, there is a non-trivial $\Lr$-definable henselian coarsening $u$ of $v$.
	\end{proof}

	\thref{prop:conj1toconj2} can be strengthened in the case that $K$ satisfies the hypothesis of \thref{lem:defvalnew}. In this case, $v$ itself is $\Lor$-definable. In the following we will argue that there are examples of strongly NIP ordered Hahn fields which do and such which do not satisfy the conditions of \thref{thm:defval}. We will use the following fact, which is established in \cite[p.~2]{halevi}.
	
	\begin{fact}\thlabel{fact:halevi}
		Let $K$ be a perfect field. Suppose that there exists a henselian valuation $v$ on $K$ such that the following hold:
		\begin{enumerate}[wide, labelwidth=!, labelindent=6pt]
			\item $v$ is defectless.
			\item The residue field $Kv$ is either an algebraically closed field of characteristic $p$ or elementarily equivalent to a local field of characteristic $0$.
			\item The ordered value group $vK$ is strongly NIP. 
			\item If $\Char(Kv) = p \neq \Char(K)$, then $[-v(p),v(p)] \subseteq pvK$.
		\end{enumerate}
		Then $K$ is strongly NIP.
	\end{fact}
	
	The following lemma will also be used in the proof of \thref{cor:classification2}. 
	
	\begin{lemma}\thlabel{lem:powstrongnip}
		Let $G$ be a strongly NIP ordered abelian group. Then the ordered Hahn field $(\R\pow{G}, <)$ is strongly NIP.
	\end{lemma}
	
	\begin{proof}
		If $K = \R\pow{G}$ is real closed, then we are done.
		Otherwise, let $v$ be the natural valuation on $K$. We will first verify that $v$ satisfies conditions (1)--(4) of \thref{fact:halevi}. Condition (4) is trivially satisfied; (2) and (3) hold by assumption. The valuation $v$ is defectless if every finite extension $(L,v)$ over $(K,v)$ is defectless. Since this always holds in the characteristic $0$ case, (1) is satisfied.
		
		Now $K$ is $\ac$-valued with angular component map $\ac : K \to \R$  given by $\ac(s) = s(v(s))$ for $s \neq 0$ and $\ac(0)=0$ (see \cite[Section~5.4~f.]{dries}). Following the argument of \cite[p.~2]{halevi}, we obtain that $(K,v,\ac)$ is a strongly NIP $\ac$-valued field. Since $\R$ is closed under square roots for positive element, for any $a \in K$ we have $a \geq 0$ if and only if the following holds in $K$:
		\[\exists y \ y^2 = \ac(a).\]
		Hence, the order relation $<$ is definable in $(K,v,\ac)$. We obtain that $(K,<)$ is strongly NIP.
	\end{proof}

	We can now give the examples announced above. By \thref{lem:powstrongnip}, it suffices to find strongly NIP ordered abelian groups $G$ which satisfy conditions (\ref{lem:defvalnew:1}) or (\ref{lem:defvalnew:2}) of \thref{lem:defvalnew} and such which do not satisfy either of the conditions.
	An explicit dp-minimal and thus strongly NIP discretely ordered abelian group is given in \thref{ex:dpmin}. \thref{ex:snipdensenotdense}~(\ref{ex:snipdensenotdense:1}) provides a strongly NIP ordered abelian group $G$ which satisfies condition (\ref{lem:defvalnew:2}) of \thref{lem:defvalnew}. \thref{ex:snipdensenotdense}~(\ref{ex:snipdensenotdense:2}) presents a strongly NIP ordered abelian group $G$ which satisfies neither condition (\ref{lem:defvalnew:1}) nor (\ref{lem:defvalnew:2}) of \thref{lem:defvalnew}. To this end, we use the characterisation of strongly NIP ordered abelian groups from \cite[Theorem~1]{halevi2}.
	
	\begin{fact}\thlabel{fact:snipgp}
		Let $G$ be an ordered abelian group. Then the following are equivalent:
		\begin{enumerate}[wide, labelwidth=!, labelindent=6pt]
			\item $G$ is strongly NIP.
			
			\item $G$ is elementarily equivalent to a lexicographic sum of ordered abelian groups $\bigoplus_{i \in I} G_i$, where for every prime $p$, we have
			$|\{i \in I \mid pG_i \neq G_i \}| < \infty$,
			and for any $i \in I$, we have
			$|\{p \text{ prime} \mid [G_i:pG_i]=\infty \}| < \infty$.
		\end{enumerate}
	\end{fact}
	
	\begin{example}\thlabel{ex:snipdensenotdense}
			\begin{enumerate}[wide, labelwidth=!, labelindent=6pt]
		\item \label{ex:snipdensenotdense:1} Let 
		$$ B = \setbr{\left. \frac{a}{p_1\ldots p_m} \ \right|\ a \in \Z, i \in \N \text{ and } p_1,\ldots,p_i\geq 3 \text{ are prime}}.$$
		$B$ is $p$-divisble for any prime $p\geq 3$. Thus,  
		$|\{p \text{ prime} \mid [G_i:pG_i]=\infty \}| = 1$. By \thref{fact:snipgp}, $B$ is strongly NIP. Moreover, it dense in its divisible hull $\Q$ but not divisible, as $\frac 1 2 \notin B$. 
		
		\item \label{ex:snipdensenotdense:2} Let $G = B \oplus \Q$ ordered lexicographically. By \thref{fact:snipgp}, $G$ is strongly NIP. However, $G$ has no limit point in $\div{G}\setminus G$, as for any $(a,b) \in \div{G} \setminus G$, there is no element in $G$ between $(a,b-1)$ and $(a,b+1)$.
		\end{enumerate}
	\end{example}

		Our next aim is to obtain a characterisation of strongly NIP almost real closed fields (see \thref{thm:arcf}).
		It is known that any almost real closed field is NIP as an ordered field (see \cite[p.~1]{jahnke3}) and we have seen in \thref{prop:dparc} that every almost real closed field with respect to a dp-minimal ordered abelian group is dp-minimal. We obtain a similar result for almost real closed fields with respect to a strongly NIP ordered abelian group.
		
	\begin{proposition}\thlabel{cor:classification2}
		Let $(K,<)$ be an almost real closed field with respect to a strongly NIP ordered abelian group and let $G$ be strongly NIP ordered abelian group. Then $(K\pow{G},<)$ is a strongly NIP ordered field.
	\end{proposition}
	
	\begin{proof}
		Let $H$ be a strongly NIP ordered abelian group such that $(K,<)$ is almost real closed with respect to $H$. As in the proof of \thref{prop:arcarc} we have that $(K\pow{G},<) \equiv (\R\pow{G \oplus H}, <)$. Since $G$ and $H$ are strongly NIP, also $G\oplus H$ is strongly NIP by \thref{fact:snipgp}. Hence, by \thref{lem:powstrongnip}, also $(K\pow{G},<)$ is strongly NIP.
	\end{proof}
		
	\begin{corollary}\thlabel{conj:classification2}
		Let $(K,<)$ be an almost real closed with respect to a henselian valuation $v$ such that $vK$ is strongly NIP. Then $(K,<)$ is strongly NIP.
	\end{corollary}
	
	\begin{proof}
		This follows immediately from \thref{cor:classification2} by setting $G = \{0\}$ and $H = vK$.
	\end{proof}

		We obtain from \thref{conj:classification2} and \thref{lem:strongnipfieldresidue} the following characterisation of strongly NIP almost real closed fields.
	
	\begin{theorem}\thlabel{thm:arcf}
	 Let $(K,<)$ be an almost real closed field with respect to some ordered abelian group $G$. Then $(K,<)$ is strongly NIP if and only if $G$ is strongly NIP.
	\end{theorem}

	\section{Concluding Remarks}\label{sec:conclusion}

	Recall our two main conjectures.
	
	\mainconjecture*
	
	\classificationconjecture*
	
	In this final section, we will show that \thref{conj:main} and \thref{conj:classification} are equivalent. 
	
	\begin{remark}\thlabel{rmk:mainrmk}
		\begin{enumerate}[wide, labelwidth=!, labelindent=6pt]
			
			\item\label{rmk:mainrmk:1} An ordered field is real closed if and only if it is o-minimal. Hence, for any real closed field $K$, if $\OO \subseteq K$ is a definable convex ring, its endpoints must lie in $K \cup \{\pm \infty\}$. This implies that any definable convex valuation ring must already contain $K$, i.e. is trivial. Thus, the two cases in the consequence of \thref{conj:main} are exclusive.
			
			\item \label{rmk:mainrmk:2} Recall from \Autoref{sec:prelim} that the field $\Q$ is not NIP. By \thref{fact:hong}, the henselian valuation $\vmin$ is $\Lr$-definable in $\Q\pow{\Z}$. Hence, by \thref{lem:strongnipfieldresidue} $(\Q\pow{\Z},<)$ is an example of an ordered field which is not real closed, admits a non-trivial $\Lor$-definable henselian valuation but is not strongly NIP.
			
		\end{enumerate}
	\end{remark}

	
	\begin{theorem}\thlabel{thm:main}
		\thref{conj:main} and \thref{conj:classification} are equivalent.
	\end{theorem}

	\begin{proof}
		Assume \thref{conj:classification}, and let $(K,<)$ be a strongly NIP ordered field which is not real closed. 
	 	Then $(K,<)$ admits a non-trivial henselian valuation $v$. By \thref{prop:conj1toconj2}, it also admits an non-trivial $\Lr$-definable henselian valuation.
		Now assume \thref{conj:main}. Let $(K,<)$ be strongly NIP ordered field. By \thref{conj2impconj1}, $K$ is almost real closed with respect to the canonical valuation $v$. 
	\end{proof}

	As a final observation, we will give two further equivalent formulations of \thref{conj:classification} which follow from results throughout this work.
	
	\begin{observation}
		The following are equivalent:
		
		\begin{enumerate}[wide, labelwidth=!, labelindent=6pt]
			
			\item\label{prop:finalobs:1} Any strongly NIP ordered field $(K,<)$ is almost real closed.
			
			\item\label{prop:finalobs:2} For any strongly NIP ordered field $(K,<)$, the natural valuation $\vnat$ on $K$ is henselian.
			
			\item\label{prop:finalobs:3} For any strongly NIP ordered valued field $(K,<,v)$, whenever $v$ is convex, it is already henselian.
		
		\end{enumerate}
		
	\end{observation}

	\begin{proof}
		(\ref{prop:finalobs:1}) implies (\ref{prop:finalobs:3}) by \thref{fact:convhens}. Suppose that (\ref{prop:finalobs:3}) holds and let $(K,<)$ be strongly NIP. By \cite[Proposition~4.2]{jahnke2}, any convex valuation is definable in the Shelah expansion $(K,<)^{\mathrm{Sh}}$, whence $(K,<,\vnat)$ is a strongly NIP ordered valued field. By assumption, $\vnat$ is henselian on $K$, which implies (\ref{prop:finalobs:2}). Finally, suppose that (\ref{prop:finalobs:2}) holds. Let $(K,<)$ be a strongly NIP ordered field and $(K_1,<)$ an $\aleph_1$-saturated elementary extension of $(K,<)$. Then $K_1\vnat = \R$. By assumption, $\vnat$ is henselian on $K_1$, whence $(K_1,<)$ is almost real closed. By \thref{prop:arcelem}, also $(K,<)$ is almost real closed.
	\end{proof}

	\section{Open Questions}\label{sec:questions}
	
	We conclude with open questions connected to results throughout this work.

 	\thref{conj:classification} for archimedean fields states that any strongly NIP ar\-chi\-me\-de\-an ordered field is real closed, as the only archimedean almost real closed fields are the real closed ones. \thref{prop:dparch} shows that any dp-minimal archimedean ordered fields is real closed. We can ask whether the same holds for all strongly NIP ordered fields.
	
	\begin{question}\thlabel{qu:sniprc}
		Let $(K,<)$ be a strongly NIP archimedean ordered field. Is $K$ necessarily real closed?
	\end{question}

		Note that any almost real closed field which is not real closed cannot be dense in its real closure (see \thref{rmk:denseimmediatefield}).		
		Thus, if \thref{conj:classification} is true, then, in particular, a strongly NIP ordered field which is not real closed cannot be dense in its real closure. Note that any dp-minimal ordered field which is dense in its real closure is real closed. Indeed, by \thref{fact:dpfield} and \thref{thm:densitytransfers}, any dp-minimal ordered field which is dense in its real closure is elementarily equivalent to a non-archimedean ordered Hahn field which is dense in its real closure. However, by \thref{rmk:denseimmediatefield} the only ordered Hahn fields which are dense in their real closure are the real closed ones.

	\begin{question}\thlabel{qu:snipdenserc}
		Let $(K,<)$ be a strongly NIP ordered field which is dense in its real closure. Is $(K,<)$ real closed?
	\end{question}

	Note that \thref{qu:snipdenserc} is more general than \thref{qu:sniprc}, as a positive answer to \thref{qu:snipdenserc} would automatically tell us that any archimedean ordered field is real closed.

		Note that any archimedean ordered field $(K,<)$ does not admit a non-trivial valuation as $\Z$ must be contained in any convex subring of $K$. Hence, any ordered field with an archimedean model does not admit a non-trivial $\Lor$-definable convex valuation. We have seen in \thref{prop:ordfieldnoarchmodel} that there are non-archimedean ordered fields which are dense in their real closure but do not have an archimedean model. In these ordered field, it may be possible to find $\Lor$-definable convex valuations. In \cite{jahnke} it is not investigated whether the two cases in \thref{prop:defval} are exclusive. We pose this as the following question. 
		
	\begin{question}
		Is there an ordered field which is dense in its real closure and admits an non-trivial $\Lor$-definable convex valuation?
	\end{question}

	In \thref{lem:lemmaerdos} and \thref{prop:lemmaerdos} we have seen that any ordered field $K$ with $\td(K) \geq \aleph_0$ there is a transcendence basis $T \subseteq K$ of $K$ which is dense in $K$. Thus, for $F = \Q(T)$, we have that $F$ is dense in $K$ and $\Q$ is relatively algebraically closed in $F$. Note that a non-archimedean ordered field $K$ with $\td(K) < \aleph_0$ cannot admit a transcendence basis dense in $K$. However, it is still possible that $K$ has a dense subfield $F$ such that $\Q$ is relatively algebraically closed in $K$. We pose this as a question for a non-archimedean real closed field with transcendece degree $1$.
	
	\begin{question}
		Let $K = \rc{\Q(t)}$ ordered by $t > \N$. Is there a dense subfield $F \subseteq K$ such that $\Q$ is relatively algebraically closed in $F$?
	\end{question}

	\begin{acknowledgements} 
		The first author was supported by a doctoral scholarship of Studienstiftung des deutschen
		Volkes as well as of Carl-Zeiss-Stiftung. 
		We started this research at the \emph{Model Theory, Combinatorics and Valued fields Trimester} at the Institut Henri Poincaré in March 2018. All three authors wish to thank the IHP for its hospitality, and Immanuel Halupczok, Franziska Jahnke, Vincenzo Mantova and Florian Severin for discussions. We also thank Yatir Halevi and Assaf Hasson for helpful comments on a previous version of this work and Vincent Bagayoko for giving an answer to a question leading to \thref{prop:ordfieldnoarchmodel}.
	\end{acknowledgements}
	
	\begin{footnotesize}
		
	\end{footnotesize}


\begin{thebibliography}{99}
			
			\bibitem{biljakovic} 
			\textsc{D.~Biljakovic}, \textsc{M.~Kochetov} and \textsc{S.~Kuhlmann},
			`Primes and irreducibles in truncation integer parts of real closed fields', 
			\textsl{Logic in Tehran}, Lecture Notes in Logic 61 (Assoc. Symbol. Logic, La Jolla, CA, 2006) 42--64.
		   
		    \bibitem{delon} 
			\textsc{F.~Delon} and \textsc{R.~Farré}, 
			`Some model theory for almost real closed fields', 
			\textsl{J. Symbolic Logic} 61 (1996) 1121--1152, 
			doi:10.2307/2275808.
					   
			\bibitem{dolich} 
			\textsc{A.~Dolich}, \textsc{J.~Goodrick} and  \textsc{D.~Lippel},
			`Dp-Minimality: Basic Facts and Examples', 
			\textsl{Notre Dame J. Form. Log.} 52 (2011) 267--288, 
			doi:10.1215/00294527-1435456.
			
			\bibitem{dries}
			\textsc{L.~van~den~Dries},
			`Lectures on the Model Theory of Valued Fields',
			\textsl{Model Theory in Algebra, Analysis and Arithmetic
			}, Lecture Notes in Mathematics 2111 (Springer, Heidelberg, 2014) 55--157, doi:10.1007/978-3-642-54936-6\_4.
			
			\bibitem{dupont} 
			\textsc{K.~Dupont}, \textsc{A.~Hasson} and \textsc{S.~Kuhlmann}, 
			`Definable Valuations induced by multiplicative subgroups and NIP Fields', \textsl{Arch. Math. Logic}, to appear, arXiv:1704.02910v3.
			
			\bibitem{engler} 
			\textsc{A.~J.~Engler} and \textsc{A.~Prestel}, 
			\textsl{Valued Fields}, 
			Springer Monographs in Mathematics (Springer, Berlin, 2005).
			
			\bibitem{erdos} 
			\textsc{P.~Erd\"os}, \textsc{L.~Gillman} and \textsc{M.~Henriksen}, 
			`An isomorphism theorem for real-closed fields', 
			\textsl{Ann. of Math. (2)} 61 (1955) 542--554, 
			doi:10.2307/1969812.
			
			\bibitem{farre} 
			\textsc{R.~Farré}, 
			`A transfer theorem for Henselian valued and ordered fields', 
			\textsl{J. Symb. Log.} 28 (1993) 915--930, 
			doi:10.2307/2275104.	
						
			\bibitem{fehm} 
			\textsc{A.~Fehm} and \textsc{F.~Jahnke}, 
			`Recent progress on definability of Henselian valuations', 
			\textsl{Ordered Algebraic Structures and Related Topics}, Contemp. Math. 697 (Amer.
			Math. Soc., Providence, RI, 2017), 135--143,
			doi:10.1090/conm/697/14049.
			
			\bibitem{halevi2} 
			\textsc{Y.~Halevi} and \textsc{A.~Hasson}, 
			`Strongly Dependent Ordered Abelian Groups and Henselian Fields', 
			\textsl{Israel J. Math.}, to appear, arXiv:1706.03376v3.
			
			\bibitem{halevi} 
			\textsc{Y.~Halevi}, \textsc{A.~Hasson} and \textsc{F.~Jahnke},
			`A Conjectural Classification of Strongly Dependent Fields', 
			\textsl{Bull. Symb. Log.}, to appear, arXiv:1805.03814v1.
			
			\bibitem{hong} 
			\textsc{J.~Hong}, 
			`Definable non-divisible Henselian valuations', 
			\textsl{Bull. Lond. Math. Soc.} 46 (2014) 14--18, 
			doi:10.1112/blms/bdt074.		
						
			\bibitem{jahnke4} 
			\textsc{F.~Jahnke} and \textsc{J.~Koenigsmann}, 
			`Definable Henselian valuations', 
			\textsl{J. Symb. Log.} 80 (2015) 85--99, 
			doi:10.1017/jsl.2014.64.	

			\bibitem{jahnke2} 
			\textsc{F.~Jahnke},
			`When does NIP transfer from fields to henselian expansions?', Preprint, 2016,  
			arXiv:1607.02953v1.	 
			
			\bibitem{jahnke3} 
			\textsc{F.~Jahnke} and \textsc{P.~Simon},
			`NIP henselian valued fields', Preprint, 2016,  
			arXiv:1606.08472v1.	 
												
			\bibitem{jahnke5} 
			\textsc{F.~Jahnke} and \textsc{J.~Koenigsmann}, 
			`Defining Coarsenings of Valuations', 
			\textsl{Proc. Edinb. Math. Soc. (2)} 60 (2017) 665--687,
			doi:10.1017/S0013091516000341.	
				
			\bibitem{jahnke} 
			\textsc{F.~Jahnke}, \textsc{P.~Simon} and \textsc{E.~Walsberg}, 
			`Dp-minimal valued fields', 
			\textsl{J. Symb. Log.} 82 (2017) 151--165, 
			doi:10.1017/jsl.2016.15.
			
			\bibitem{johnson} 
			\textsc{W.~Johnson}, 
			`On dp-minimal fields', Preprint, 2015, 
			arXiv:1507.02745v1.
			
			\bibitem{knebusch} 
			\textsc{M.~Knebusch} and \textsc{M.~J.~Wright}, 
			`Bewertungen mit reeller Henselisierung', 
			\textsl{J. Reine Angew. Math.} 286/287 (1976) 314--321, 
			doi:10.1515/crll.1976.286-287.314.
			
			\bibitem{kuhlmann2} 
			\textsc{F.-V.~Kuhlmann},
			`Dense subfields of Henselian fields, and integer parts', 
			\textsl{Logic in Tehran}, Lecture Notes in Logic 61 (Assoc. Symbol. Logic, La Jolla, CA, 2006) 204--226.
			
			\bibitem{kuhlmann} 
			\textsc{S.~Kuhlmann}, 
			`Ordered Exponential Fields', 
			\textsl{Fields Inst. Monogr.} 12 (Amer. Math. Soc., Providence, RI, 2000), 
			doi:10.1090/fim/012.
					
			\bibitem{onshuus} 
			\textsc{A.~Onshuus} and \textsc{A.~Usvyatsov},
			`On dp-minimality, strong dependence and weight', 
			\textsl{J. Symb. Log.} 76 (2011) 737--758, 
			doi:10.2178/jsl/1309952519.
			
			\bibitem{prestel} 
			\textsc{A.~Prestel}, 
			`Definable Henselian valuation rings', 
			\textsl{J. Symb. Log.} 80 (2015) 1260--1267, 
			doi:10.1017/jsl.2014.52.	
					
			\bibitem{robinson}
			\textsc{A.~Robinson} and \textsc{E.~Zakon}, 
			`Elementary properties of ordered abelian groups', 
			\textsl{Trans. Amer. Math. Soc.} 96 (1960) 222--236, 
			doi:10.1090/S0002-9947-1960-0114855-0.
			
			\bibitem{robinson2}
			\textsc{J.~Robinson}, 
			`Definability and decision problems in arithmetic', 
			\textsl{J. Symbolic Logic} 14 (1949) 98--114, 
			doi:10.2307/2266510.
					
						
			\bibitem{shelah2} 
			\textsc{S.~Shelah}, 
			`Dependent first order theories, continued', 
			\textsl{Israel J. Math.} 173 (2009) 1--60, 
			doi:10.1007/s11856-009-0082-1.	
			
			\bibitem{shelah} 
			\textsc{S.~Shelah}, 
			`Strongly dependent theories', 
			\textsl{Israel J. Math.} 204 (2014) 1--83, 
			doi:10.1007/s11856-014-1111-2.	
				
			\bibitem{simon} 
			\textsc{P.~Simon}, 
			`A Guide to NIP Theories', 
			\textsl{Lecture Notes in Logic} 44 (ASL, Cambridge University Press, Cambridge, 2015), 
			doi:10.1017/CBO9781107415133.
			
			\bibitem{zakon}
			\textsc{E.~Zakon}, 
			`Generalized archimedean groups', 
			\textsl{Trans. Amer. Math. Soc.} 99 (1961) 21--40, 
			doi:10.1090/S0002-9947-1961-0120294-X.

				
		\end{thebibliography}
\end{document}